\documentclass{article}
\usepackage[utf8]{inputenc}

\usepackage{amssymb}
\usepackage{amsthm}
\usepackage{amsmath,amscd}
\usepackage[mathscr]{euscript}
\usepackage[all]{xy}
\usepackage{lmodern}
\usepackage[T1]{fontenc}
\usepackage[textwidth=14cm,hcentering]{geometry}
\usepackage[colorlinks=true,linkcolor=red,citecolor=blue]{hyperref}
\usepackage{cleveref}
\usepackage{mathtools}
\usepackage{xspace}
\usepackage{tikz}
\usetikzlibrary{cd}
\usetikzlibrary{arrows,positioning}

\title{On endomorphisms of topological Hochschild homology}
\author{Maxime Ramzi}
\date{}

\newtheorem{thm}{Theorem}[section]
\newtheorem{lm}[thm]{Lemma}
\newtheorem{prop}[thm]{Proposition}
\newtheorem{cor}[thm]{Corollary}

\newtheorem*{thm*}{Theorem}

\theoremstyle{definition}
\newtheorem{defn}[thm]{Definition}
\newtheorem{cons}[thm]{Construction}
\newtheorem{assu}[thm]{Assumption}
\newtheorem{nota}[thm]{Notation}

\newtheorem{ex}[thm]{Example}
\newtheorem{exs}[thm]{Examples}
\newtheorem{exc}[thm]{Exercise}
\newtheorem{rmk}[thm]{Remark}
\newtheorem{ques}[thm]{Question}
\newtheorem{conj}[thm]{Conjecture}
\newtheorem{warn}[thm]{Warning}
\newtheorem{obs}[thm]{Observation}
\AtBeginEnvironment{defn}{\pushQED{\qed}}
\AtEndEnvironment{defn}{\popQED\enddefn}
\AtBeginEnvironment{cons}{\pushQED{\qed}}
\AtEndEnvironment{cons}{\popQED\endcons}
\AtBeginEnvironment{assu}{\pushQED{\qed}}
\AtEndEnvironment{assu}{\popQED\endassu}
\AtBeginEnvironment{nota}{\pushQED{\qed}}
\AtEndEnvironment{nota}{\popQED\endnota}
\AtBeginEnvironment{conv}{\pushQED{\qed}}
\AtEndEnvironment{conv}{\popQED\enddefn}\AtBeginEnvironment{ex}{\pushQED{\qed}}
\AtEndEnvironment{ex}{\popQED\endex}
\AtBeginEnvironment{exs}{\pushQED{\qed}}
\AtEndEnvironment{exs}{\popQED\endexs}
\AtBeginEnvironment{exc}{\pushQED{\qed}}
\AtEndEnvironment{exc}{\popQED\endexc}
\AtBeginEnvironment{rmk}{\pushQED{\qed}}
\AtEndEnvironment{rmk}{\popQED\endrmk}
\AtBeginEnvironment{ques}{\pushQED{\qed}}
\AtEndEnvironment{ques}{\popQED\endques}
\AtBeginEnvironment{conj}{\pushQED{\qed}}
\AtEndEnvironment{conj}{\popQED\endconj}
\AtBeginEnvironment{warn}{\pushQED{\qed}}
\AtEndEnvironment{warn}{\popQED\endwarn}
\AtBeginEnvironment{obs}{\pushQED{\qed}}
\AtEndEnvironment{obs}{\popQED\endobs}

\newtheorem{thmx}{Theorem}

\newcommand{\THH}{\mathrm{THH}}
\newcommand{\Oo}{\mathcal{O}}

\newcommand{\op}{^{\mathrm{op}}}
\newcommand{\cat}{\mathrm}
\newcommand{\Cat}{\cat{Cat}}
\newcommand{\on}{\operatorname}
\newcommand{\id}{\mathrm{id}}
\newcommand{\Fun}{\on{Fun}}
\newcommand{\map}{\on{map}}
\newcommand{\Map}{\on{Map}}

\newcommand{\NN}{\mathbb N}
\newcommand{\Z}{\mathbb Z}
\newcommand{\Sph}{\mathbb S}
\newcommand{\CMon}{\mathrm{CMon}}

\newcommand{\Ss}{\mathcal{S}}
\newcommand{\Sp}{\cat{Sp}}

\newcommand{\PrL}{\cat{Pr}^\mathrm{L} }

\newcommand{\Alg}{\mathrm{Alg}}
\newcommand{\CAlg}{\mathrm{CAlg}}

\newcommand{\Mod}{\cat{Mod}}

\newcommand{\cn}{\mathrm{cn}}

\newcommand{\Cof}{\mathrm{Cof}}

\newcommand{\HH}{\mathrm{HH}}
\newcommand{\Perf}{\mathbf{Perf}}
\newcommand{\perf}{\mathrm{perf}}

\newcommand{\Ind}{\mathrm{Ind}}
\newcommand{\End}{\mathrm{End}}
\newcommand{\eend}{\mathrm{end}}

\newcommand{\pt}{\mathrm{pt}}
\newcommand{\colim}{\mathrm{colim}}
\newcommand{\fib}{\mathrm{fib}}

\newcommand{\add}{\mathrm{split}}

\newcommand{\dbl}{{\mathrm{dbl}}}

\newcommand{\st}{{\mathrm{st}}}
\newcommand{\ext}{\mathrm{ex}}
\newcommand{\Cob}{\mathrm{Cob}}

\newcommand{\cyc}{\mathrm{cyc}}

\newcommand{\one}{\mathbf{1}}

\newcommand{\Frrig}{\mathrm{Fr}^{\mathrm{rig}}}

\newcommand{\C}{\mathbf{C}}
\newcommand{\E}{\mathcal E}
\newcommand{\EE}{\mathbb E}
\newcommand{\F}{\mathbb F}

\newcommand{\cC}{\mathscr{C}}

\newcommand{\tr}{\mathrm{tr}}
\newcommand{\Tr}{\mathrm{Tr}}

\newcommand{\Mot}{\mathrm{Mot}}
\newcommand{\loc}{\mathrm{loc}}

\newcommand{\colax}{\mathrm{colax}}

\newcommand{\Day}{\mathrm{Day}}

\newcommand{\TR}{\mathrm{TR}}
\newcommand{\Lace}{\mathrm{Lace}}

\newcommand{\cospan}{\mathrm{coSpan}}
\newcommand{\fin}{\mathrm{fin}}

\DeclareFontFamily{U}{min}{}
\DeclareFontShape{U}{min}{m}{n}{<-> udmj30}{}

\newcommand{\category}{\xspace{$\infty$-category}\xspace}
\newcommand{\categories}{\xspace{$\infty$-categories}\xspace}
\newcommand{\operad}{\xspace{$\infty$-operad}\xspace}

\begin{document}

\maketitle
\begin{abstract}
    We compute endomorphisms of topological Hochschild homology ($\THH$) as a functor on stable $\infty$-categories, as well as variants thereof: we also compute endomorphisms of the $k$-linear Hochschild homology functor $\HH_k$ over some base $k$; and endomorphisms of $\THH$ as a functor on stably symmetric monoidal $\infty$-categories.
\end{abstract}
\section*{Introduction}\label{sec:introchEnd}
One important aspect in the study of cohomology theories is the study of their self-operations - for example, the Adams operations on topological $K$-theory helped simplify the solution of the Hopf invariant one problem \cite{adamsatiyah}. 

Beyond algebraic topology, one can also consider cohomology theories on schemes, and there again it is a reasonable question to ask what their operations are - as in topology, the Adams operations on algebraic $K$-theory provide a wealth of extra structure which can help in its study. A particularly rich source of cohomology theories in algebraic geometry is the theory of \emph{localizing invariants}, which are some kind of noncommutative cohomology theories, and for those we may try to compute (noncommutative) operations. Chief among those is algebraic $K$-theory, whose universality makes the calculation of noncommutative operation essentially trivial (or rather, reduces it to simply computing $K$-theory itself).

The goal of this paper is to study the endomorphisms of topological Hochschild homology, $\THH$, in various guises. Already in \cite{loday}, Loday constructed a number of operations on Hochschild homology of rational commutative rings, and used it to construct filtrations and refine the study of rational Hochschild homology.

Similar questions have also been studied by Wahl--Westerland \cite{wahlwesterland}, Wahl \cite{wahl} and Klamt \cite{klamt}. In these works, the authors study what they call ``formal operations'', which can be viewed as an approximation to the spectrum of operations, of Hochschild homology relative to a base commutative ring $R$, evaluated on dg algebras over $R$. While we are not able to access endomorphisms of $\THH$ as a functor on \emph{ring spectra}, we do give a complete description of the actual endomorphism spectrum of $\THH$ as a functor of \emph{able \categories}, as well as of the endomorphism monoid of $\THH$ as a symmetric monoidal functor. 

As a plain functor, the answer looks as follows:
\begin{thmx}\label{thm:endplain}
    As a plain functor $\THH: \Cat^\perf\to \Sp$, the $S^1$-action induces an equivalence $$\Sph[S^1]\simeq \eend(\THH)$$ 
As a symmetric monoidal functor, there is an equivalence $$\Map_{\CAlg(\Sp)}(\Sph^{S^1}, \Sph)\simeq \End^\otimes(\THH)$$ and the space $\Map_{\CAlg(\Sp)}(\Sph^{S^1}, \Sph)$ can be described\footnote{But it is not exactly $S^1$.}. 
\end{thmx}
We also prove analogous results in relative contexts, that is, for $\THH$ relative to a base commutative ring spectrum $k$. Thus, in the completely noncommutative setting, there seem to be no interesting operations besides the well-known circle action. The key tool to proving this result is (a slight extension of) the Dundas-McCarthy theorem, expressing $\THH$ in terms of algebraic $K$-theory, where the operations are easy to understand. 

We also study analogous questions in multiplicative settings, related to the more general questions studied in the works of Wahl--Westerland, Wahl and Klamt mentioned above.  Specifically, we use these to study the endomorphism spectrum of $\THH$ viewed as a functor on $\Alg_\Oo(\Cat^\perf)$ for some (single-colored) \operad $\Oo$. The integral answer there seems to be more subtle, and we only get partial results, but we also get a full description of these endomorphisms if one suitably localizes $\THH$, e.g. rationally. The following is our second main result, where $L$ denotes either rationalization, or $T(n)$-localization for some height $n\geq 1$ and implicit prime $p$ - here, $\gamma_k$ denotes the (unique up to conjugacy) length $k$ cycle in $\Sigma_k$, and $\Oo(k)^{\gamma_k}$ is the space of $\Z$-fixed points of $\Oo(k)$ with automorphism given by the action of $\gamma_k$ (it retains a $C_k$-action, where $C_k$ is the centralizer of $\gamma_k$ in $\Sigma_k$, in this case the subgroup generated by $\gamma_k$):
\begin{thmx}\label{thm:opns}
    
For any single-colored \operad $\Oo$, there is a canonical map $$\bigoplus_{k\geq 1} \Sph[(\Oo(k)^{\gamma_k} \times S^1)_{hC_k}]\to \eend_{\Fun(\Alg_\Oo(\Cat^\perf),\Sp)}(\THH)$$ such that : 
\begin{enumerate}
\item For any finite set $S\subset \mathbb N_{\geq 1}$, the restriction to $\bigoplus_{k\in S}$ admits a splitting; 
\item The induced map on $L$-localization, specifically: 
$$L(\bigoplus_{k\geq 1} \Sph[(\Oo(k)^{\gamma_k} \times S^1)_{hC_k}])\to \eend_{\Fun(\Alg_\Oo(\Cat^\perf),\Sp)}(L\THH)$$ 
is an equivalence. 
\end{enumerate}

\end{thmx}
We also prove, by studying the ``spectrum of cyclotomic Frobenii'', that the relevant map is \emph{not} an equivalence in the integral case (though this is not visible in \cite{klamt} precisely because the cyclotomic Frobenius does not exist over $\Z$). 

In particular, rationally and for $\Oo$ being the commutative operad, we find that endomorphisms of $\HH_\mathbb Q$ on stably symmetric monoidal rational \categories are simply given by $\bigoplus_{n\geq 1} \mathbb Q[S^1/C_n]$, corresponding to Loday's shuffle operations, proving that the operations Loday constructed not only extend from commutative dgas over $\mathbb Q$ to rational symmetric monoidal stable \categories, but also that in this bigger world, they are essentially the \emph{only} operations. 

Beyond the methodology of \Cref{thm:endplain} and formal nonsense, the key input for \Cref{thm:opns} is the calculation of $\THH$ of groupoid-indexed colimits from \cite{CCRY}, generalizing Blumberg-Cohen-Schlichtkrull's calculation of $\THH$ of Thom spectra. 
\subsection*{Outline}
In \Cref{section:day} we set up some preliminaries regarding Day convolution that will be convenient for the later sections. In \Cref{section:plain}, we apply these results to compute plain endomorphisms of $\THH$, and we study the relative setting in \Cref{section:endbase}. Finally, in \Cref{section:opsO} we initiate the study of operations for stable \categories equipped with multiplicative structures. 

\Cref{ch:motsplit} serves as recollections on what we call splitting invariants/motives (typically called additive motives in the literature), and \Cref{app:acttr} is used for some elementary calculations of certain cyclic group actions on specific traces in symmetric monoidal \categories. 

\subsection*{Relation to other work}
We learned that Sasha Efimov had independently discovered \Cref{thm:endplain}, with essentially the same proof \cite{youtubeSasha}, though with a slight inaccuracy in the multiplicative version of the statement, and with an added difficulty in the proof coming from staying within localizing motives throughout the proof. 

Li and Mondal also consider a similar question about de Rham cohomology in \cite{limondal}, which is closely related to Hochschild homology. At this point, the author does not know a way of relating the two sets of results precisely, though a comparison would surely be interesting. 

Finally, we point out that this paper is essentially a chapter in the author's PhD thesis \cite{Thesis}. In particular, I use a version of the Dundas-McCarthy theorem from an earlier chapter of said thesis, which has not appeared as a preprint yet, but Harpaz, Nikolaus and Saunier have also proved essentially the same version of this theorem in \cite{Victor1}, so that we will try to refer to their work whenever possible. However, there are two points where we do need the slightly stronger results from \cite{Thesis}, namely to identify more general (first) Goodwillie derivatives, and to identify certain $S^1$-actions (or more precisely, certain $C_n\subset S^1$-actions). The required results from \cite{Thesis} will appear in preprint form, hopefully soon, but since the thesis was written, I was able to answer certain questions raised therein that pertain to the relevant chapter, and will therefore need some extra time to record the proofs. 

\subsection*{Conventions}
We work with $\infty$-categories throughout, as developed by Lurie in \cite{HTT,HA}. Categorical notions (functors, subcategories etc.) are to be understood in this context unless explicitly stated. 

\begin{itemize}
    \item $\Ss$ denotes the \category of spaces\footnote{Or homotopy types, $\infty$-groupoids, anima,...}, $\Sp$ the \category of spectra, $\PrL$ that of presentable \categories and left adjoints between them, as well as its  stable variant $\PrL_\st$.$\Cat^\ext$ denotes the \category of small stable \categories and exact functors, and $\Cat^\perf$ the full subcategory thereof spanned by idempotent-complete \categories. We let $\Fun^\ext$ denote the \category of exact functors. 

\item Throughout, $\Map$ denotes mapping spaces, while $\map$ is reserved for mapping spectra in stable \categories, and similarly for $\End$ and $\eend$. We use $\hom$ for internal homs in closed monoidal \categories. 
    \item We recall in \Cref{ch:motsplit} our conventions regarding localizing invariants and splitting\footnote{More often called ``additive''.} invariants. 
    \item The symbol $K$ denotes nonconnective algebraic $K$-theory, while $K^\cn$ denotes its connective variant. Note that, unlike in \cite{BGT13}, we do not require the latter to be invariant under idempotent-completion, and in particular, $K^\cn$ is only the connective cover of $K$ for idempotent-complete stable \categories (otherwise it agrees with it in degrees $\geq 1$). 
     \item Given a space $X$, $LX$ denotes $\map(S^1,X)$, its free loop space.  
\end{itemize}
\subsection*{Acknowledgements}
This paper is essentially Chapter 5 of my PhD thesis, and I wish to thank Ben Antieau and Sasha Efimov for their proofreading of it,and  my advisors Jesper Grodal and Markus Land for their help and constant support throughout my PhD, as well as their enthusiasm.  

I also want to thank Vova Sosnilo and Christoph Winges - our collaboration on \cite{RSW} is what inspired me to think about other applications of the Dundas-McCarthy theorem; and Nathalie Wahl for helpful discussions surrounding her related work as well as Klamt's work on the matter. 

I've had helpful conversations about the contents of this work with Sasha Efimov, Achim Krause, Ishan Levy, Zhouhang Mao, Thomas Nikolaus, Victor Saunier and Lior Yanovski. 

Some of this work was written while I was visiting various universities in North America, and I wish to thank them and the people who invited me or housed me there for their hospitality: in chronological order, those were Mike Hopkins at Harvard University, Ben Antieau and Achim Krause at the Institute for Advanced Study, David Gepner at Johns Hopkins Univerisity, Niranjan Ramachandran at the University of Maryland, Akhil Mathew at the University of Chicago, Paul Balmer and Morgan Opie at UCLA, and Elden Elmanto at the University of Toronto. These visits were made possible by the EliteForsk Rejsestipendium I received, and I gratefully acknowledge the support of Denmark's Uddannelses- og Forskningsministeriet, as well as Jesper for getting me this grant. 

While much of this research was conducted, I was furthermore supported by the Danish National Research Foundation through the Copenhagen Centre for Geometry and Topology (DNRF151), and in the finishing stages of writing this version of the article, I am funded by the Deutsche Forschungsgemeinschaft (DFG, German Research Foundation) -- Project-ID 427320536 -- SFB 1442, as well as by Germany's Excellence Strategy EXC 2044 390685587, Mathematics Münster: Dynamics--Geometry--Structure.

Finally, thanks to q.uiver for help with the commutative diagrams. 
\newline 
\tableofcontents

\section{Day convolution of localizing invariants}\label{section:day}
Since $\THH$ admits a multiplication, one way to compute endomorphisms of $\THH$ is to compute $\THH$-linear maps out of ``$\THH\otimes\THH$'' into $\THH$, for some meaning of $\otimes$. The relevant meaning here is the one that encodes the symmetric monoidal structure on the functor $\THH$, as opposed to that on each of its values\footnote{Indeed, on $\THH$ of an arbitrary stable \category, there is no such structure !}. The relevant tensor product for this is Day convolution, which is known to encode symmetric monoidal structures. We will not enter into too many details about it, but there is also a version of the Day tensor product for localizing invariants, which is better suited for our purposes. 

The goal of this section is therefore to prove the following, which will be the key tool to study endomorphisms of $\THH$: 
\begin{prop}\label{prop:P1astensorTHH}
    Let $\E$ be a cocomplete stable \category and $E:\Cat^\perf\to \E$ be an accessible localizing invariant. The Day convolution \emph{as localizing invariants} $\THH\otimes^\Day E$ of $\THH$ with $E$ is naturally equivalent to $C\mapsto P_1E^{\mathrm{lace}}(C,\id_C)$.
\end{prop}
Here $E^{\mathrm{lace}}$ is defined in \cite[Definition 3.10]{Victor1} (simply replace $K$ by $E$ in that definition), and it is what is denoted $E^\cyc$ in \cite{Thesis}. 
\begin{rmk}
We stress that the usual Day convolution of $\THH$ and $E$ need not be a localizing invariant. What we mean by ``the Day convolution \emph{as localizing invariants}''  is what is obtained by first taking ordinary Day convolution and then taking the universal localizing invariant with a map from this Day convolution. Through the \category of localizing motives $\Mot_\loc$, we may view this as ``exactified Day convolution'', as discussed in \Cref{cons:exactDay}.     
\end{rmk}
\begin{rmk}\label{rmk:TrThyconvolution}
    One can state and prove a refined version of this result - implying in particular an $S^1$-equivariant version of it, as both sides have a natural $S^1$-action. However, it requires setting up more technicalities, more than seem worth it for our purposes, so we leave the refined statement for the future. 
\end{rmk}

To prove this, we use the following preparatory lemma: 
\begin{lm}
Let $C$ be a presentably symmetric monoidal \category, and let $x\in C$. For any accessible\footnote{Here accessibility is just used to deal with size issues, namely to make sure that the Day tensor product is defined \emph{a priori}.} functor $F: C\to \E$ to a cocomplete \category $\E$, there is a natural equivalence $$\Map(x,-)\otimes^\Day F \simeq F(\hom(x,-))$$ 
\end{lm}
\begin{proof}
Let $G$ be any functor $C\to \E$, and let $\mu: C\times C\to C$ denote the tensor product, $\mathrm{pr}_i: C\times C\to C$ the two projection maps. By definition of the Day tensor product, there is an equivalence $$\Map(\Map(x,-)\otimes^\Day F, G)\simeq \Map(\Map(x,\mathrm{pr}_1)\boxtimes F\circ \mathrm{pr}_2, G\circ \mu)$$ 
where we temporarily use $\boxtimes$ to distinguish the external tensor product from the Day tensor product.

By currying, this is equivalent to $$\Map(\Map(x,-), \Map(F, G\circ \mu))$$
where the target is $y\mapsto \Map(F,G(y\otimes -))$. 

By the Yoneda lemma, this is equivalent to $\Map(F,G(x\otimes -))$ and by adjunction, this is equivalent to $\Map(F(\hom(x,-)), G)$, as was to be proved. 
\end{proof}
\begin{cons}\label{cons:exactDay}
    If $C, \E$ are stable and $F,G$ are exact functors, $F\otimes^\Day G$ need not be exact a priori. However, one can then take its first derivative to make it exact, and this induces a Day convolution monoidal structure on $\Fun^{\ext,\mathrm{acc}}(C,\E)$ - see e.g. \cite[Theorem 4.5]{horel} in the case where the source is small, but the generalization to accessible functors is straightforward. We call it the \emph{exactified} Day convolution. 
\end{cons}
In that case, there is an analogue of the previous lemma when we consider instead the mapping \emph{spectrum} out of $x$, and the exactified Day convolution. The proof is the same\footnote{Or in fact it can be seen to follow from the previous lemma, as $\map(x,-) = P_1\Sph[\Map(x,-)]$.}, so we only state it:

\begin{cor}\label{cor:stableDay}
    Let $C$ be a stable presentably symmetric monoidal \category, and $x\in C$ be an object admitting internal homs $\hom(x,-)$. For any accessible exact functor $F:~C\to~\E$ to a cocomplete \emph{stable} \category $\E$, there is a natural equivalence $\map(x,-)\otimes^\Day~F\simeq~F(\hom(x,-))$, where now $\map$ denotes the mapping spectrum functor, and $\otimes^\Day$ denotes the exactified Day convolution product.
\end{cor}

To apply this to $\THH$, we need to recall a modern version of the Dundas-McCarthy theorem \cite{DM}. Our reference for this is \cite{Victor1}, though a variant can also be found in the author's thesis manuscript, \cite{Thesis}\footnote{Though until that is available as a preprint, we will stick to \cite{Victor1}.}. To state it, we introduce a bit of notation. 

Notice that for $C\in\Cat^\perf$, $\Lace(C,\Sigma^n)$ can be described as $\Fun^\ext(A_n,C)$ for some $A_n\in\Cat^\perf$, where $\Lace$ is as in \cite{Victor1}. Specifically, $A_n$ is easily verified to be $\Perf(\Sph[\sigma_{-n}])$, where $\Sph[\sigma_{-n}]$ is the free associative algebra on a class in degree $-n$. Indeed, a map $x\to \Sigma^nx$ is the same as a map $\Omega^n x\to x$, which is the same as a map $\Sph[\sigma_{-n}]\to \map(x,x)$. 
\begin{nota}
    Let $A_n :=\Perf(\Sph[\sigma_{-n}])$ as discussed above. 
\end{nota}
With this notation, we can state \cite[Theorem 4.42]{Victor1} as follows:
\begin{cor}\label{cor:presTHH}
    There is a natural equivalence $$\THH\simeq\colim_n \Omega^n \tilde K(\Fun^\ext(A_n, -))$$
\end{cor}
\begin{rmk}\label{rmk:robustDM}
    Since We are taking iterated loop spaces, for any fixed $k$, the $k$th homotopy group of the right hand side only depends on the connective part of $K$-theory for $n$ large enough, and hence in the colimit. Thus this corollary is independent of whether we consider connective or nonconnective $K$-theory. 
\end{rmk}
If we rephrase it by viewing $\THH$ as a functor on $\Mot_\add$, the \category of splitting motives\footnote{These are typically called ``additive motives'' in the literature. This, especially the corresponding notion of ``additive invariant'' sounds very confusing and less descriptive than ``splitting'', so we have opted for this change of name. We have gathered a few helpful facts in \Cref{ch:motsplit}.}, we can interpret it as follows with  a direct application of \Cref{cor:mapinmot}:
\begin{cor}\label{cor:presTHHsplit}
    Considering $\THH$ as a functor on $\Mot_\add$, letting $\tilde{\mathcal A}_n$ denote $\mathcal U_\add(A_n)/\mathcal U_\add(Sp^\omega)$, we have:
    $$\THH \simeq \colim_n \Omega^n\map(\tilde{\mathcal A}_n,-)$$
\end{cor}
Thus, directly using the formula from \Cref{cor:stableDay}, we find:
\begin{cor}\label{cor:thhasP1add}
    Let $E:\Mot_\add\to \E$ be an exact accessible functor to a cocomplete stable \category. With exactified Day convolution, we have an equivalence: 
    $$\THH\otimes^\Day E\simeq P_1E^{\mathrm{lace}}$$
\end{cor}
\begin{proof}
    Using \Cref{cor:stableDay} and \Cref{cor:presTHHsplit}, we find $$\THH\otimes^\Day E\simeq \colim_n \Omega^n \tilde E(\Fun^\ext(A_n, -))\simeq \colim_n\Omega^n \tilde E(\Lace(-;\Sigma^n)) =: P_1E^{\mathrm{lace}}$$ 
\end{proof}

\begin{proof}[Proof of \Cref{prop:P1astensorTHH}] 
We first note that the functor $\Mot_\add\to \Mot_\loc$ witnessing that every localizing invariant is a splitting invariant is a localization.

Thus, if $E,F$ are localizing invariants such that their Day convolution \emph{as splitting invariants} is already localizing, it is also their Day convolution \emph{as localizing invariants}. 

We also note that if $E$ happened to be a localizing invariant, then by \cite[Proposition 5.15]{Victor1} $P_1E^{\mathrm{lace}}$ is already localizing. Thus in this situation, the calculation from \Cref{cor:thhasP1add} works for Day convolution \emph{as splitting invariants} or \emph{as localizing invariants}, which is what we wanted to prove. 
\end{proof}
We point out another, more naive consequence of \Cref{cor:presTHH}, which can sometimes be convenient:
\begin{cor}\label{cor:coder}
    Let $E$ be a splitting invariant. There is an equivalence:$$\map(\THH,E)\simeq \lim \Sigma^n \tilde E(A_n)$$
\end{cor}
\begin{proof}
    By adjunction, $\map(K(\Fun^\ext(A_n,-)), E)\simeq \map(K, E(A_n\otimes -))$ and so, by the universal property of $K$-theory from \cite{BGT13}, the latter is $E(A_n)$. 

    Putting things together as $n$ varies gives the result. 
\end{proof}
\begin{rmk}
    By \Cref{rmk:robustDM}, we can use either connective or nonconnective $K$-theory here, which is why $E$ is allowed to be a splitting invariant as opposed to a localizing invariant. 
\end{rmk}
So while tensoring with $\THH$ acts as a Goodwillie derivative, mapping out of $\THH$ acts as some kind of \emph{co}-derivative. 

\begin{cor}\label{cor:modthh}
   Evaluation at $\Sp^\omega$ induces a symmetric monoidal equivalence $$\Mod_\THH(\Fun^{\loc,\omega}(\Cat^\perf,\Sp))\simeq \Sp$$
\end{cor}
\begin{proof}
   Evaluation at $\Sp^\omega$ is represented by $\THH$ on this \category, so this functor admits a left adjoint given by $X\mapsto X\otimes\THH$. 

    Furthermore, evaluation at $\Sp^\omega$ preserves filtered colimits and hence $\THH$ is a compact object in that module \category. Thus, since $$\eend_\THH(\THH) \simeq \Map(K,\THH)\simeq \Sph$$ this left adjoint is fully faithful. 

Now for any finitary localizing invariant $E$, $E\otimes^\Day \THH \simeq P_1E^{\mathrm{lace}}$ by \Cref{prop:P1astensorTHH} and by \cite[Theorem 4.2.3]{Thesis} the latter is $X_E\otimes\THH$ for some fixed spectrum $X_E$ (note that we are not considering the trace theory structure here, so $X_E$ is simply a spectrum, no $S^1$-action in sight). 

It follows that the image of the left adjoint $\Sp\to \Mod_\THH(\Fun^{\loc,\omega}(\Cat^\perf,\Sp))$ contains all the free modules, i.e. contains a collection of generators under colimits. Since it is fully faithful and colimit-preserving, it follows that it is an equivalence.

Both sides are symmetric monoidal \categories, and $\Sp$ has a unique symmetric monoidal structure with $\Sph$ as the unit, so the symmetric monoidal claim follows as well. 
\end{proof}

\section{Plain endomorphisms of $\THH$}\label{section:plain}
In this section, we prove \Cref{thm:endplain}, that is, we compute endomorphisms of $\THH$ as a plain functor, as well as its endomorphisms as a lax symmetric monoidal functor. Using a extensions-restriction of scalars adjunction, this will ultimately boil down to a computation of $\THH\otimes^\Day\THH$, which the previous section has given us the tools to study. 

Thus we begin with the following well-known computation about $\THH$ - to limit confusion, in the statement and its proof, we let $\mathrm{T}$ denote the lace invariant ``topological Hochschild homology'' (or trace theory, in the language of \cite{Thesis},\cite{OWRTC}) and $\THH$ denote the corresponding localizing invariant. :
\begin{prop}\label{prop:P1THH}
    There is an equivalence of lace invariants with $S^1$-action: $$P_1\THH^{\mathrm{lace}}\simeq \mathrm{T}^{S^1}$$ where the $S^1$-action on the right is coinduced. 
\end{prop}
To prove this, we will briefly use the results from \cite{Thesis}.
\begin{proof}
By \cite[Theorem 4.2.11]{Thesis}, it suffices to evaluate at $(\Sp,\id_\Sp)$. This is now a classical computation (more generally in the case of ring spectra), cf. e.g. \cite[Proposition 3.2]{LarsTC}. There, the result is $\Sph[S^1]\otimes \mathrm{T}$, which differs from our claim by a homological shift, but the shift here comes from the shift in the equivalence $\Lace(\Sp^\omega,\Sigma^{n+1})\simeq \Perf(\Sph\oplus\Sigma^n\Sph)$, or more generally, $\Lace(\Perf(A),\Sigma M\otimes_A -)\simeq \Perf(A\oplus M)$ for $(A,M)$ connective, cf. \cite[Proposition 3.2.2, Lemma 3.4.1]{raskin}. 

Ultimately, the point as explained in \textit{loc. cit.} is that the (homogeneous) degree $1$ part of the cyclic bar construction computing $\THH(\Sph\oplus\Sigma^{n-1}\Sph)$ is a free cyclic object on the (non-cyclic) bar construction computing $\THH(\Sph,\Sigma^{n-1}\Sph)$, and so its realization has an induced $S^1$-action on $\Sph$. There is a shift involved, and the relation $\Sigma \Sph^{S^1}\simeq \Sph[S^1]$ gets us to the announced result. 

Alternatively, one can use the general formula for $\THH$ of square zero extensions which we recall later in \Cref{prop:thhsqz}. 
\end{proof}
\begin{cor}\label{cor:thhsquared}
  There is an $S^1$-equivariant, $\THH$-linear equivalence $$\THH\otimes^\Day \THH\simeq \THH^{S^1}$$ where: 
  \begin{itemize}
      \item $\THH\otimes^\Day \THH$ is the extension of scalars to $\THH$ of $\THH$ with its $S^1$-action;
      \item $\THH^{S^1}$ is the coinduced $S^1$-action on the unit in $\THH$-modules.
  \end{itemize} 
\end{cor}
\begin{rmk}
    The same refinement as in \Cref{rmk:TrThyconvolution} would allow us to make this equivalence $S^1\times S^1$-equivariant, but we will not need this, and as mentioned there, the technicalities needed for this statement seem to outweigh the benefits. 
\end{rmk}
\begin{proof}
By \Cref{prop:P1astensorTHH} and \Cref{prop:P1THH}, we know that there is an equivalence in $\Fun^{\loc,\omega}(\Cat^\perf,\Sp)^{BS^1}$ between the two. 

By \Cref{cor:modthh}, the equivalence $\Mod_\THH(\Fun^{\loc,\omega}(\Cat^\perf,\Sp))^{BS^1}\simeq\Sp^{BS^1}$ is given by evaluation at $\Sp^\omega$, and therefore factors through the forgetful functor $$\Mod_\THH(\Fun^{\loc,\omega}(\Cat^\perf,\Sp))^{BS^1}\to \Fun^{\loc,\omega}(\Cat^\perf,\Sp)^{BS^1}.$$ Thus there is an equivalence between the two, as claimed.
\end{proof}
In fact, this can be upgraded to a multiplicative equivalence: 
\begin{cor}\label{cor:thhsquaredmult}
    There is an equivalence of commutative $\THH$-algebras with $S^1$-action $$\THH\otimes^\Day\THH \simeq \THH^{S^1}$$
\end{cor}
\begin{proof}
    By passing to $\Sp$ under the equivalence from \Cref{cor:modthh}, with $A~=~\THH\otimes^\Day~\THH$ (or more precisely its image in $\Sp$), we have the following situation: 
    \begin{enumerate}
        \item There is a commutative algebra $A\in \CAlg(\Sp^{BS^1})$; 
        \item Its underlying object $UA\in\Sp^{BS^1}$ is coinduced, specifically of the form $\Sph^{S^1}$;
        \item There is a map $A\to \Sph$ in $\CAlg(\Sp)$.
    \end{enumerate}
    The algebra map $A\to\Sph$ comes from the multiplication map $$\THH\otimes^\Day \THH\to \THH$$ We claim that this is enough to conclude $A\simeq \mathrm{coInd}^{S^1}\Sph$ as commutative algebras.  

Since $A\to \Sph$ is an algebra map,  it is split, and hence, on underlying objects $\Sph\oplus\Omega\Sph\to \Sph$ must be the projection onto the first factor (there are no nonzero maps $\Omega\Sph\to \Sph$). Thus, as a map $\Sph^{S^1}\to \Sph$ it is the co-unit map of the co-induction adjunction. 

It follows that the coinduction map $UA\to \Sph^{S^1}$ is an equivalence. But this coinduction map can be made into a commutative algebra map, since the forgetful functor preserves limits and hence coinductions. Since it is also conservative, we find that $A\simeq \Sph^{S^1}$.
\end{proof}
\begin{rmk}
    This proof actually gives a bit more than what we claim: it shows that specifically, the map of commutative $\THH$-algebras with $S^1$-action $$\THH\otimes^\Day\THH\to \THH^{S^1}$$ induced by adjunction from the multiplication $\THH\otimes^\Day\THH\to \THH$ is itself an equivalence. 
\end{rmk}
\begin{rmk}
    It is overwhelmingly likely that by setting up everything (trace theories, derivatives etc.) symmetric monoidally from the start, one could obtain a cleaner proof of the previous two results by applying only improved versions of \Cref{prop:P1astensorTHH} and \Cref{prop:P1THH}.
\end{rmk}

We can now prove \Cref{thm:endplain}. In fact, we prove more. Let us first deal with the non-symmetric monoidal part. 
\begin{thm}\label{thm:endplainpowers}
The $S^1$-action on $\THH$ induces an equivalence $$\Sph[S^1]\simeq \eend(\THH)$$

More generally, as functors $\Cat^\perf\to \Sp$, for any $p,q\geq 1$, we have $\map(\THH^{\otimes p},\THH^{\otimes q})\simeq 0$ if $p\neq q$, and $\Sph[((S^1)^p \times\Sigma_p)]$ if $p=q$. Here, the tensor powers are pointwise.
\end{thm}
\begin{rmk}
    We could also explain how to compute endomorphisms of Day-tensor powers of $\THH$, but they follow immediately from the case of $\THH$ itself together with \Cref{cor:thhsquared}, so we do not delve into it in more detail. 
\end{rmk}
\begin{proof}
    
For the first part of the statement, note that $$\map(\THH,\THH)\simeq \map_\THH(\THH\otimes^\Day \THH,\THH)\simeq \map_\THH(\THH^{S^1},\THH)$$
$$\simeq \map_\THH(\THH,\THH)\otimes S^1 \simeq \THH(\Sph)\otimes S^1 \simeq \Sph[S^1]$$ 

 More generally, this shows that for any spectrum $X$, the canonical assembly map $$X\otimes~\map(\THH,\THH)\to~\map(\THH, X\otimes~\THH)$$ is an equivalence.

Now we move on to pointwise tensor products. The point here is that by the main result of \cite{Motloc}, recalled in \Cref{thm:motlocisloc}, we can compute this mapping spectrum as a mapping spectrum between functors defined on $\Mot_\loc$, where $\THH$ is exact. Indeed, this main result states that $\Cat^\perf\to \Mot_\loc$ is a Dwyer-Kan localization, and $\THH^{\otimes p}$ clearly factors through it, hence the mapping spectra agree\footnote{$\THH$ itself is a localizing invariant, so by the very definition of $\Mot_\loc$, this is obvious for $\THH$. On the other hand, $\THH^{\otimes p}$ is not localizing: it's not even an additive functor, so we do need something else.}.

Note that since $\THH$ is exact on $\Mot_\loc$, $\THH^{\otimes p}$ is $p$-homogeneous thereon, and so this deals automatically with the case $q<p$. 

We now deal with the case $q=p$: in that case, we can use the classification of $p$-homogoneous functors: $$\map(\THH^{\otimes p},\THH^{\otimes p})\simeq \map_{\Sigma_p}(\bigoplus_{\Sigma_p}\THH^{\boxtimes p},\bigoplus_{\Sigma_p}\THH^{\boxtimes p})\simeq \map(\THH^{\boxtimes p},\bigoplus_{\Sigma_p}\THH^{\boxtimes p})$$ 
where $\boxtimes$ indicates an external products, so we are considering (symmetric) functors with $p$ inputs.

Now using currying and the fact that the assembly map $$X\otimes \map(\THH,\THH)\to \map(\THH,X\otimes \THH)$$ is an equivalence, we find the desired result. 

In the case $p<q$, we appeal to \Cref{lm:homogeneousinduced} below.

Note that in all cases, we find that for any $X$, the canonical map $$X\otimes\Map(\THH^{\otimes p}, \THH^{\otimes q})\to \Map(\THH^{\otimes p}, X\otimes\THH^{\otimes q}) $$ is an equivalence.  
\end{proof}
\begin{lm}\label{lm:homogeneousinduced}
Let $C,D$ be stable \categories and $F:C^n \to D$ be a functor which is exact in each variable. For $m<n$ and for any $m$-excisive functor $G:C\to D$, $\map(G,F\circ \delta_n) = 0$. 

In other words, if $D$ admits sequential limits, the $m$th co-Goodwillie derivative $P^m(F\circ \delta_n)$ vanishes. 
\end{lm}
\begin{proof}
By adjunction, we have $\map(G,F\circ \delta_n)\simeq \map(G\circ \prod_n, F)$. 

The reduction of $G\circ \prod_n$ in the sense of \cite[Construction 6.1.3.15]{HA} is, by definition (cf. \cite[Construction 6.1.3.20]{HA}) the $n$th cross effect of $G$. It follows from \cite[Remark 6.1.3.23]{HA} that $\mathrm{Red}(G\circ \prod_n) =: \mathrm{cr}_n(G) = \mathrm{cr}_n(P_nG)= P_{1,...,1}\mathrm{cr}_nG = 0 $, where the second equality comes from $G$ being $m$- and hence $n$-excisive. It is $0$ because $G$ is also $P_{n-1}G$, as $m\leq n-1$. 

Now as a functor of $n$-variables, $F$ is $(1,...,1)$-homogeneous in the sense of \cite[Definition 6.1.3.7]{HA}, and so, to prove the claim, it suffices to prove that there is an equivalence $$\map(G\circ \prod_n, F)\simeq \map(\mathrm{Red}(G\circ \prod_n),F)$$
because the latter is also equivalent to $\map(P_{1,...,1}\mathrm{Red}(G\circ \prod_n),F)=  \map(0,F)=0$. 

This now follows from \cite[Lemma B.1]{heutsgoodwillie} : \textit{loc. cit.} implies the third equality in $$\mathrm{Red}(G\circ \prod_n)\simeq \mathrm{Red}(G\circ \coprod_n) =: \mathrm{cr}_n(G)\simeq \mathrm{cr}^n(G) := \mathrm{coRed}(G\circ \prod_n)$$
and $\mathrm{coRed}$ is left adjoint to the inclusion of reduced functors, cf. \cite[Proposition 6.2.3.8]{HA}. 
\end{proof}
\begin{rmk}
    The failure of this result when $F\circ \delta_n$ is replaced with an arbitrary $n$-homogeneous functor comes from the fact that such a functor is of the form $(F~\circ~\delta_n)_{h\Sigma_n}$ for some symmetric $F$, and we cannot ``pull out'' the $_{h\Sigma_n}$.
    
    Let us give an interpretation in terms of co-Goodwillie derivatives which will come up again later in \Cref{section:opsO}. For simplicity we focus on $m=1$. In this case, $P^1(F\circ \delta_n)_{h\Sigma_n}$ is always very explicit: it is given by $X\mapsto \lim_k \Sigma^k F\circ \delta_n(\Omega^k X) $. 

    It turns out that each of the maps $\Sigma^{k+1} (F\circ \delta_n)(\Omega^{k+1} X)\to \Sigma^k (F\circ \delta_n)(\Omega^k X)$ is nullhomotopic, as they are essentially given by the diagonal $S^1\to S^n$. But they are not $\Sigma_n$-\emph{equivariantly} nullhomotopic (they are essentially Euler maps), and so taking orbits \emph{can} break the $0$-ness of the limit, and this is indeed what often happens (e.g. in \Cref{thm:extraops}). 
\end{rmk}
We also have all the tools at hand to prove the multiplicative part of \Cref{thm:endplain}: 
\begin{proof}[Proof of the multiplicative part of \Cref{thm:endplain}]
Using the equivalence between commutative algebras for the Day tensor product and lax symmetric monoidal functors, we can rewrite $\End^\otimes(\THH)$ as $\Map_\CAlg(\THH,\THH)$. 

By adjunction, and by \Cref{cor:thhsquaredmult}, this is the same as $$\Map_{\CAlg(\THH)}(\THH\otimes^\Day\THH,\THH) \simeq \Map_{\CAlg(\THH)}(\THH^{S^1},\THH)$$
Finally, using \Cref{cor:modthh}, the latter is equivalent to $\Map_{\CAlg(\Sp)}(\Sph^{S^1},\Sph)$, as claimed. In the proposition below, we describe the latter mapping space as claimed in \Cref{thm:endplain} - let us point out that this proof no longer has anything to do with $\THH$. 
\end{proof}
\begin{prop}\label{prop:map}
    The space $\Map_{\CAlg(\Sp)}(\Sph^{S^1},\Sph)$ is a disjoint union of circles. Its $\pi_0$ is $\widehat\Z$, the profinite integers. 
\end{prop}
The proof of this fact relies on two facts, an easy one in rational homotopy theory, and Mandell's theorem in $p$-adic homotopy theory:
\begin{lm}\label{lm:freerat}
    Let $R$ be a commutative $\mathbb Q$-algebra in $\Sp$. The commutative $R$-algebra $R^{S^1}$ is free on a generator in degree $-1$. 
\end{lm}
\begin{proof}
It suffices to prove it over $\mathbb Q$. As a $\mathbb Q$-module, $\mathbb Q^{S^1}\simeq \mathbb Q\oplus\Omega\mathbb Q$, so it receives a map from the free commutative $\mathbb Q$-algebra on $\Omega \mathbb Q$. The latter is, as a $\mathbb Q$-module, $\bigoplus_n ((\Omega \mathbb Q)^{\otimes n})_{h\Sigma_n}$. Now as a $\Sigma_n$-module, $(\Omega \mathbb Q)^{\otimes n}$ is a copy of the sign representation in degree $-n$, which has no coinvariants and no higher homology (as we are over $\mathbb Q$ and $\Sigma_n$ is finite), and hence simply vanishes. 
\end{proof}
And now Mandell's theorem: 
\begin{thm}[Mandell]\label{thm:Mandell}
    Let $X$ be a finite nilpotent space. The canonical map induces an equivalence $$X_p\simeq \Map_{\CAlg(\overline{\F_p})}(\overline{\F_p}^X,\overline{\F_p})$$ 
    where $X_p$ denotes the $p$-completion of $X$. 
\end{thm}
\begin{cor}\label{cor:Mandell}
    Let $X$ be a finite nilpotent space. There is an equivalence $$L(X_p)\simeq \Map_{\CAlg(\F_p)}(\F_p^X,\F_p)$$
    where $LY:= \Map(S^1,Y)$ is the free loop space on $Y$. 
\end{cor}
\begin{proof}
    This follows from \Cref{thm:Mandell} by Galois descent, as the Galois action on $\Map_{\CAlg(\overline{\F_p})}(\overline{\F_p}^X,\overline{\F_p})$ is trivial on $X_p$\footnote{It is trivial because the map from $X_p$ to the mapping space in \Cref{thm:Mandell} is an assembly map which is clearly Galois equivariant.} and thus the fixed points for it are just $L(X_p)$. 
\end{proof}
\begin{rmk}\label{rmk:calgname}
    For $X = S^1$ (which is the universal case up to $p$-adic completion), this implies that $\pi_0\Map_{\CAlg(\F_p)}(\F_p^{S^1},\F_p)\cong\Z_p$. The author does not really know how to describe these morphisms in a concrete way, how to ``name'' them. This lack of concreteness permeates through the later calculations and so we are essentially unable to ``name'' the elements of $\pi_0\Map(\Sph^{S^1},\Sph)\cong \widehat\Z$ except for $0$.
\end{rmk}
\begin{proof}[Proof of \Cref{prop:map}]
    We use the arithmetic fracture square for $\Sph$. Writing $\mathbb A_f~:=~(\prod_p\mathbb Z_p)_\mathbb Q$ for the finite adèles, we have a pullback square of commutative algebras: 
    \[\begin{tikzcd}
	\Sph & {\prod_p\Sph_p} \\
	{\mathbb Q} & {\mathbb A_f}
	\arrow[from=1-1, to=2-1]
	\arrow[from=2-1, to=2-2]
	\arrow[from=1-2, to=2-2]
	\arrow[from=1-1, to=1-2]
\end{tikzcd}\]
so that, using adjunction in each of the mapping spaces, we have a pullback square 
\[\begin{tikzcd}
	{\Map(\Sph^{S^1},\Sph)} & {\prod_p\Map(\Sph_p^{S^1},\Sph_p)} \\
	{\Map(\mathbb Q^{S^1},\mathbb Q)} & {\Map(\mathbb A_f^{S^1},\mathbb A_f)}
	\arrow[from=1-1, to=2-1]
	\arrow[from=2-1, to=2-2]
	\arrow[from=1-2, to=2-2]
	\arrow[from=1-1, to=1-2]
\end{tikzcd}\]
By \Cref{lm:freerat}, the bottom map is simply $B\mathbb Q \to B\mathbb A_f$.

By \cite[Theorem 2.4.9]{DAG13}, basechange further induces an equivalence\footnote{See \cite[Corollary 7.6.1]{yuanFrob} to see how to deduce this from the given reference.} $\Map(\Sph_p^{S^1},\Sph_p)\simeq~\Map(\F_p^{S^1},\F_p)$ which, by \Cref{cor:Mandell} is $L(S^1)_p\simeq \coprod_{\mathbb Z_p}(S^1)_p$. 

Using basechange along $\Z_p\to W(\overline \F_p)$ one finds that for a given $p$, all the maps $(S^1)_p\to B\mathbb A_f$ indexed over $\mathbb Z_p$ are the same map. Thus we find that the pullback can be computed as: 
\[\begin{tikzcd}
	{\Map(\Sph^{S^1},\Sph)} & {\prod_pL(B\mathbb Z_p)} \\
	{B\mathbb Z} & {\prod_pB(\mathbb Z_p)} \\
	{B\mathbb Q} & {B\mathbb A_f}
	\arrow[from=1-2, to=2-2]
	\arrow[from=2-2, to=3-2]
	\arrow[from=3-1, to=3-2]
	\arrow[from=2-1, to=3-1]
	\arrow[from=2-1, to=2-2]
	\arrow[from=1-1, to=2-1]
	\arrow[from=1-1, to=1-2]
\end{tikzcd}\]

The top right vertical map is a fold map of the form $\coprod_{\prod_p \Z_p} \prod_p B\mathbb Z_p\to \prod_p B\mathbb Z_p$ by distributivity of products over coproducts, and hence the top left vertical map, as a pullback thereof, is also a fold map and hence $\map(\Sph^{S^1},\Sph)\simeq \coprod_{\prod_p \Z_p}B\mathbb Z$ . The claim then follows as $\widehat\Z\cong \prod_p\Z_p$. 
\end{proof}

\begin{rmk}
The author is not sure whether this is the right perspective to have, but the ``weird result'' coming out of this calculation seems to stem from the following strikingly different behaviour of certain cochain algebras over $\mathbb Q$ and over $\F_p$: in characteristic $0$, the presentation of $C^*(X;R)$ as an $\EE_\infty$-ring is essentially independent of $R$'s arithmetic, while in characteristic $p$, properties such as algebraic closure seem to affect the presentation quite a bit.
\end{rmk}
\section{Endomorphisms over a base}\label{section:endbase}
In this section, we extend the calculation from \Cref{thm:endplainpowers} to $\THH(-/k)$, that is, $\THH$ relative to a base commutative ring spectrum $k$. A surprising feature of this calculation is that while for classical rings such as $k=\F_p$, one can phrase the question entirely algebraically, the answer turns out to involve $\THH(k)$. We only state and prove the following result for plain endomorphisms, but the method is simple enough that the other variants from the previous section can also easily be calculated:
\begin{cor}\label{cor:endplaink}
Let $k$ be a commutative ring spectrum. When regarded as a functor $\Cat^\perf_k\to~\Sp$, the endomorphism spectrum of $\THH(-/k)$ is $$\map_{\THH(k)}(k,k)[S^1]$$
When regarded as a functor $\Cat^\perf_k\to \Mod_k$, it has the following endomorphism $k$-algebra: 
$$\map_{k\otimes \THH(k)}(k,k)[S^1]$$
\end{cor}
\begin{rmk}
    We provide both calculations here, for the following reason: on the one hand, when $k$ is a classical ring, the second endomorphism spectrum is clearly a purely algebraic gadget, as both $\Cat^\perf_k$ and $\Mod_k$ are; and this showcases the appearance of $\THH(k)$ in this purely algebraic question. 

    On the other hand, one could argue that the $k$-linear structure is ``overdetermined'' by putting it both in the source and the target which could explain the appearance of $\THH$ - e.g., as it does in $\Fun^L(\Mod_\Z,\Mod_\Z)$ as opposed to $\Fun^L(\Mod_\Z,\Sp)$. The first calculation shows that this is not the case: even removing this over-determinacy, the answer involves $\THH(k)$. 
\end{rmk}
\begin{proof}
    We prove the second variant, the first variant works the same by replacing the word ``$(k,\THH(k))$-bimodule'' with ``right $\THH(k)$-module''.

    The idea for this calculation is that $\THH(-/k)\simeq k\otimes_{THH(k)}\widetilde{\THH}$, where $$\widetilde\THH: \Cat^\perf_k\to \Mod_{\THH(k)}$$ is $\THH$ of the underlying stable \category equipped with its $\THH(k)$-action; and that this expression is well-defined for any $(k,\THH(k))$-bimodule in place of $k$. 
    
    So let $M$ be such a bimodule, and let us instead try to prove  more generally that there is an equivalence $$\Map_k(M\otimes_{\THH(k)}\widetilde{\THH},\THH(-/k))\simeq \Map_{k\otimes\THH(k)}(M,k)[S^1]$$ 

    For this, we note that $M\mapsto M\otimes_{\THH(k)}\widetilde{\THH}$ is a functor, and thus (using the $S^1$-action on $\THH(-/k)$) we have a canonical map
    $$\Map_{k\otimes\THH(k)}(M,k)[S^1]\to \Map_k(M\otimes_{\THH(k)}\widetilde{\THH},\THH(-/k))$$
    and the claim is that this map is an equivalence. But now both sides clearly send colimits in the $M$ variable to limits, and so we are reduced to the case $M~=~k~\otimes~\THH(k)$\footnote{For the first variant, we are reduced to the right $\THH(k)$-module $\THH(k)$ itself. }, where this map is the canonical map $$k[S^1]\to \Map_\Sp(\THH\circ U_k, \THH(-/k))$$ where $U_k:\Cat^\perf_k\to \Cat^\perf$ is the forgetful functor. 

    By adjunction nonsense, this last mapping spectrum is simply $$\Map_\Sp(\THH, \THH(-/k)\circ (k\otimes -))\simeq \Map_\Sp(\THH, k\otimes\THH)\simeq k[S^1]$$ by \Cref{thm:endplainpowers} (and the claim, in the proof, that $\Map(\THH,X~\otimes~\THH)~\simeq~X~\otimes~\Map(\THH,\THH)$).

    One checks, e.g. using $k$-linearity and compatibility with the $S^1$-actions, that the map $k[S^1]\to k[S^1]$ thus obtained is the identity, and thus we are done. 
\end{proof}
\begin{ex}
    Let $k$ be a (classical) perfect field of characteristic $p$. In this case, by Bökstedt periodicity (see e.g. \cite{KNBök}), we have a presentation of $k$ as a $\THH(k)$-module, or as a $k\otimes\THH(k)$-module, so we can actually compute these mapping spectra quite explicitly.
    
    Indeed, letting $\sigma \in\pi_2\THH(k)$ denote the Bökstedt element, we have $\pi_*\THH(k)\cong k[\sigma]$, a polynomial algebra. Thus $k\simeq \THH(k)/\sigma$ as a $\THH(k)$-module, and so $$\map_{\THH(k)}(k,k)[S^1]\simeq \fib(k\xrightarrow{0}\Omega^2k)[S^1] \simeq (k\oplus\Omega^3 k)[S^1]$$
    and we find, on top of the expected $k[S^1]$, a degree $-3$-generator (plus the $S^1$-action). 

    The calculation as bimodule is only a bit more involved: $$\map_{k\otimes \THH(k)}(k,k)\simeq \map_k(k\otimes_{k\otimes \THH(k)}k,k)$$
    Using the universal property of $\THH(k)$ as a commutative algebra, we find that $k\otimes_{k\otimes\THH(k)}k \simeq \THH(k)\otimes ( k\otimes_{\THH(k)}k)$, and using now Bökstedt's equivalence $\THH(k)\simeq k[\Omega S^3]$, we find $k\otimes_{\THH(k)} k\simeq k[S^3]$, so that, in total, 
    $$\map_{k\otimes \THH(k)}(k,k)\simeq \prod_n (\Omega^{2n} k\oplus \Omega^{2n+3}k)$$
    To get $k$-linear endomorphisms of $\HH_k$, one simply adjoins $[S^1]$ to this. 
\end{ex}
\begin{rmk}\label{rmk:hhkname}
    In the previous examples, it is not clear to the author what  the extra operations ``do''. Similarly to the situation in \Cref{rmk:calgname}, we do not know how to ``name'' them, and it would probably be worthwhile to spend time figuring out what these operations do, or are. 
\end{rmk}

\section{Operations on $\THH$ of $\Oo$-algebras}\label{section:opsO}
The goal of this section is to initiate the study of endomorphisms of $\THH$ viewed as a functor $\Alg_\Oo(\Cat^\perf)\to \Sp$, that is, to study the extra operations that arise on $\THH(C)$ when $C$ has a particular kind of multiplicative structure encoded by a one-colored \operad $\Oo$. Our main result in this section is \Cref{thm:opns}.

As is clear from the statement of \Cref{thm:opns}, the case of endomorphisms of $\THH$ viewed as a functor on $\Alg_\Oo(\Cat^\perf)$ is more subtle.

We first describe the general approach to this question, and then specialize to get the precise results that we claimed.

We let $U:\Alg_\Oo(\Cat^\perf)\to \Cat^\perf$ denote the forgetful functor, with a left adjoint $F$ such that $UF\simeq \bigoplus_n (\Oo(n)\otimes(-)^{\otimes n})_{h\Sigma_n}$ \cite[Proposition 3.1.3.13]{HA}. 

The idea now is that by general adjunction nonsense, $$\Map_{\Fun(\Alg_\Oo(\Cat^\perf),\Sp)}(\THH \circ U, \THH \circ U)\simeq \Map_{\Fun(\Cat^\perf,\Sp)}(\THH, \THH \circ UF)$$
and so we are ``left with'' understanding $\THH \circ UF$. 

For this, we compute each of the individual terms of $$\THH\circ UF = \bigoplus_n \THH((\Oo(n)\otimes (-)^{\otimes n})_{h\Sigma_n})$$ 
\begin{prop}\label{prop:thhcolimtext}
 Let $O$ be a space with a $\Sigma_n$-action. There is an equivalence, natural in $C\in\Cat^\perf$: $$\THH((O\otimes C^{\otimes n})_{h\Sigma_n})\simeq \bigoplus_{\sigma\in\Sigma_n/\mathrm{conj}}(O^\sigma \otimes \THH(C)^{\otimes n(\sigma)})_{hC(\sigma)}$$ where for $\sigma\in\Sigma_n$, $n(\sigma)$ is the number of cycles appearing in $\sigma$, $C(\sigma)$ is the centralizer of $\sigma$ in $\Sigma_n$ and $O^\sigma = L(O_{h\Sigma_n})\times_{LB\Sigma_n}\{\sigma\}$ with its residual $C(\sigma)$-action\footnote{Here, $L$ denotes the free loop space, i.e. $\Map(S^1,-)$.}.  
\end{prop}
\begin{rmk}
    The action of $C(\sigma)$ on $\THH(C)^{\otimes n(\sigma)}$ is a bit subtle to describe. Part of the point here is that there is a functor $\Sp^{BS^1}\to \Sp^{BC(\sigma)}$ of the form $X\mapsto X^{\otimes n(\sigma)}$ such that the above is obtained by applying it to $\THH(C)$ with its $S^1$-action. 

    In extreme cases, the action is easier to describe: if $\sigma= 1$ is the trivial permutation, $n(\sigma)=~n, C(\sigma) =~\Sigma_n$, and the action is simply given by the permutation action. If $\sigma$ is the\footnote{There is only one up to conjugacy.} length $n$ cycle, $n(\sigma) =1$ and $C(\sigma)$ is generated by $\sigma$, i.e. a cyclic group of order $n$, and the action in this case is given by restricting the $S^1$-action. 
\end{rmk}
Given the above remark, we cannot be too precise about the $C(\sigma)$-action for general $\sigma\in\Sigma_n$, and the specifics we will need about this formula are the following:
\begin{prop}\label{prop:thhcolimweak}
    Let $O$ be a space with a $\Sigma_n$-action. There is an equivalence, natural in $C\in\Cat^\perf$: $$\THH((O\otimes C^{\otimes n})_{h\Sigma_n})\simeq (O^{C_n}\otimes\THH(C))_{hC_n}\oplus F(C)$$ 
    where the $C_n$-action on $O^{C_n}\otimes \THH(C)$ is diagonal, the one on $\THH(C)$ being restricted from $S^1$, and such that $C\mapsto F(C)$ factors through $\Mot_\loc$ and is a (finite) direct sum of $k$-homogeneous functors, $1<k\leq n$. 
\end{prop}

The main tool for this proposition is \cite[Corollary 7.15]{CCRY}, which we recall for convenience of the reader. We specialize it to the relevant context. In the notation of that paper, this means setting $\mathscr{C} = \Sp$. 
\begin{rmk}
Note that in \cite{CCRY}, we use the trace functor $\Tr$ as in \cite{HSS}. This is defined on $\Cat^{\mathrm{lace}}$ in the notation of \cite{Victor1}, and is a trace-like exact functor, whose value at $(\Sp,\id_\Sp)$ is $\Sph$. Thus, by the universal property of $\THH$ proved in \cite{Victor1}, there is a canonical map $\THH\to \Tr$ inducing an equivalence on $(\Sp,\id_\Sp)$. Thus, it is an equivalence everywhere, see e.g. \cite[Proposition 4.2.8]{Thesis} (the proof adapts readily to the context of \cite{Victor1}).

See also \cite[Corollary 4.2.67]{Thesis} for a more structured comparison between $\Tr$ and $\THH$.  
\end{rmk}
\begin{defn}
    Let $A$ be a space and $\zeta: A\to \PrL_\st$ be a functor with values in dualizable objects\footnote{For our purposes, the reader can replace this with compactly generated \categories, or equivalently think of a functor $A\to\Cat^\perf$.}, and let $\chi_\zeta:LA\to \Sp$ denote the functor $$\gamma\mapsto \THH(\zeta(\gamma(0)); \zeta\circ \gamma)$$ 
\end{defn}
\begin{thm}[{\cite[Corollary 7.15]{CCRY}}]\label{thm:CCRY}
      Let $A$ be a space and $\zeta: A\to \PrL_\st$ be a functor with values in dualizable objects. The colimit $\colim_A \zeta$ is still dualizable, and there is a natural equivalence: $$\THH(\colim_A \zeta)\simeq\colim_{LA} (\chi_\zeta)$$
\end{thm}

We need two more trace-calculations. The first is the trace of a permutation: 
\begin{lm}\label{lm:trperm}
Fix $\sigma\in\Sigma_n$ for some $n\geq 0$, and let $\cC$ be a symmetric monoidal \category and $x\in\cC$ a dualizable object. There is an equivalence $$\tr(x^{\otimes n}; \sigma)\simeq \dim(x)^{\otimes n(\sigma)}$$ where $n(\sigma)$ is the number of cycles in $\sigma$, in $\End(\one_\cC)$. 

    Furthermore, if $\sigma$ is a length $n$ cycle, the $C(\sigma) = \langle \sigma\rangle$-action on the right, induced by functoriality in $LB\Sigma_n$, is mapped to the $C_n\subset S^1$-action on $\dim(x)$ (which is $\dim(x)^{\otimes n(\sigma)}$ since $n(\sigma)=1$). 
\end{lm}
\begin{proof}
First, we note that the pair $(x^{\otimes n},\sigma)$ decomposes as a tensor product over the cycles in $\sigma$ of terms of the form $(x^{\otimes \ell}, C_\ell)$ for each cycle $C_\ell$ in $\sigma$\footnote{Though this decomposition is not natural in $\sigma \in LB\Sigma_n$, there we would have to worry about cycle types too.}. Therefore it suffices to identifty $\tr(x^{\otimes n};C_n)$ where $C_n$ is the length $n$ cycle.

By \cite{harpazcob}\footnote{Harpaz gave a full proof of the one dimension cobordism hypothesis, previously sketched in all dimensions by Lurie in \cite{luriecob}.}, the one-dimensional cobordism \category $\Cob_1$ is the universal symmetric monoidal \category on a dualizable object, thus the calculation of the trace and the $C(\sigma)$ -action can be performed entirely there. 

Let $1^+$ denote the universal dualizable object in $\Cob_1$, $1^-$ its dual, and $n^+ := (1^+)^{\otimes n}$ (similarly for $n^-$). Recall that $\End(\one_{\Cob_1})= \coprod_n (BS^1)^n_{h\Sigma_n}$ is given by the space of $1$-dimensional, closed oriented manifolds, i.e. disjoint unions of oriented circles. So, to identify the object $\tr(n^+,C_n)$ as a manifold it suffices to identify its number of connected components. 

Recall that the cobordism representing the coevaluation looks as follows (we read our cobordisms from left to right):

\includegraphics[scale=0.05]{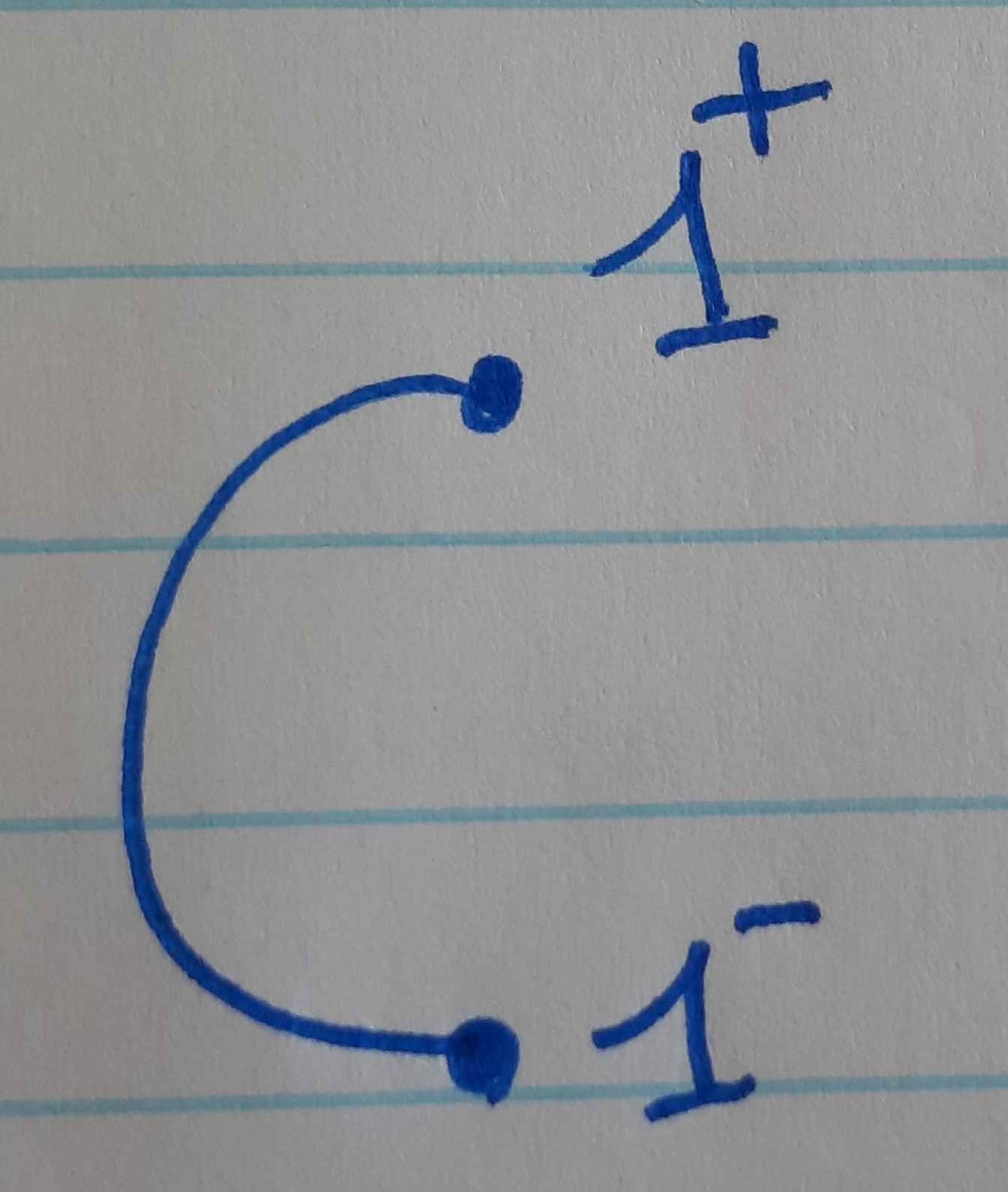}

and the cobordism representing the evaluation looks like:

\includegraphics[scale=0.04]{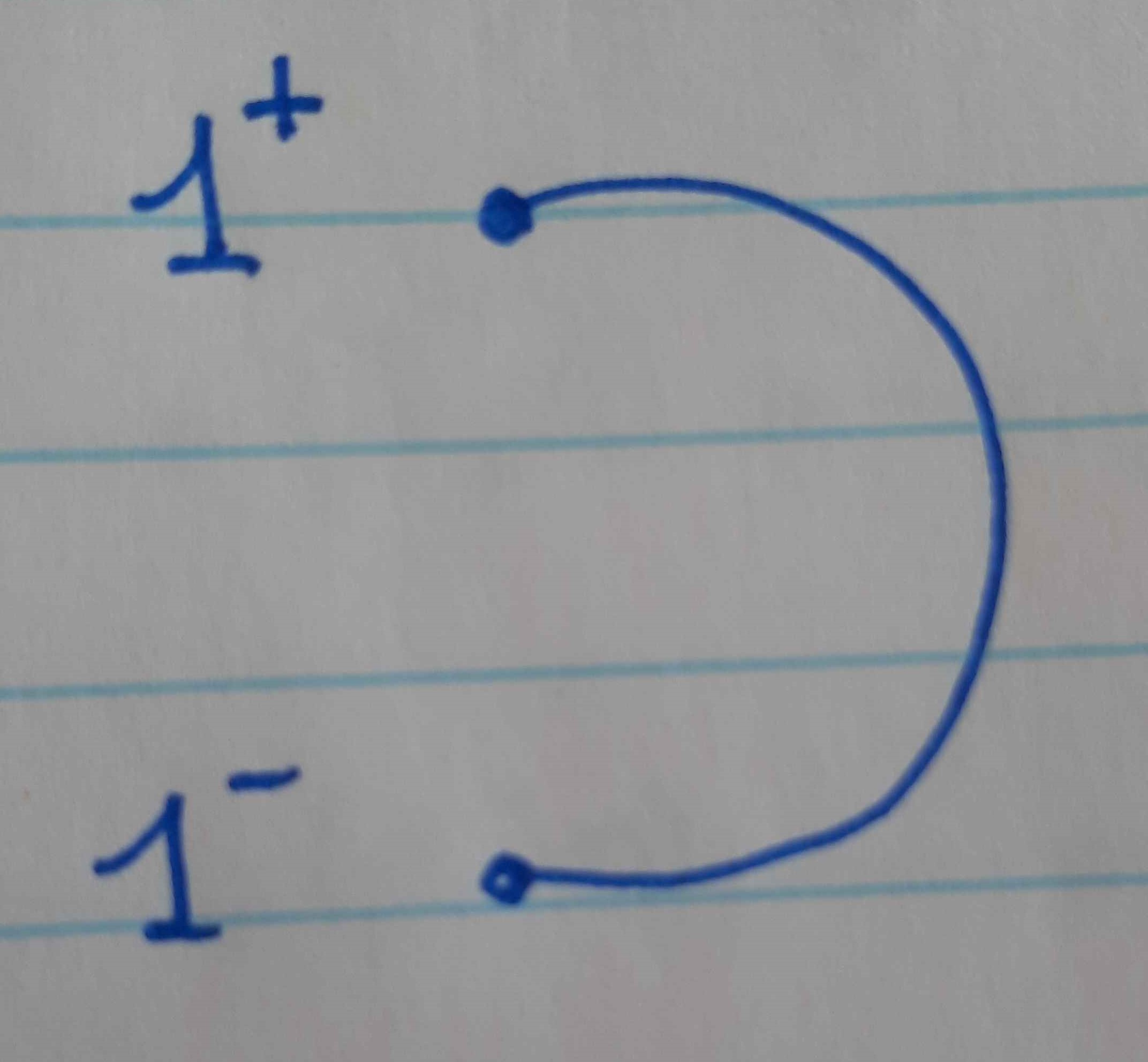}

The manifold representing $\tr(n^+,C_n)$ is obtained by taking the cobordism representing $C_n$ from $n^+$ to itself, the one representing $\id$ from $n^-$ to itself, and adding evaluations and coevaluations. 

We let $(i^{\pm },0), i\in\{0,n-1\}$ denote the endpoints in the source, $(i^{\pm },1)$ denote the endpoints in the target. By definition, $(i^+,0)$ is related by the cobordism to $(\sigma(i)^+,1)$; this is in turn related by the evaluation to $(\sigma(i)^-,1)$, which is related by the $\id$-cobordism to $(\sigma(i)^-, 0)$ which, finally, is related by the coevaluation to $(\sigma(i)^+,0)$. 

Since $\sigma= C_n$ is transitive (by definition of cycle!) it follows that all the points $(i^+,0)$ are connected. Continuing this argument shows that all points are connected, and thus $\tr(n^+,C_n)\in\End(\one_{\Cob_1})$ is a connected manifold, i.e. the circle, i.e. $\dim(1^+)$. 

Here is a pictorial example: 

\includegraphics[scale=0.1]{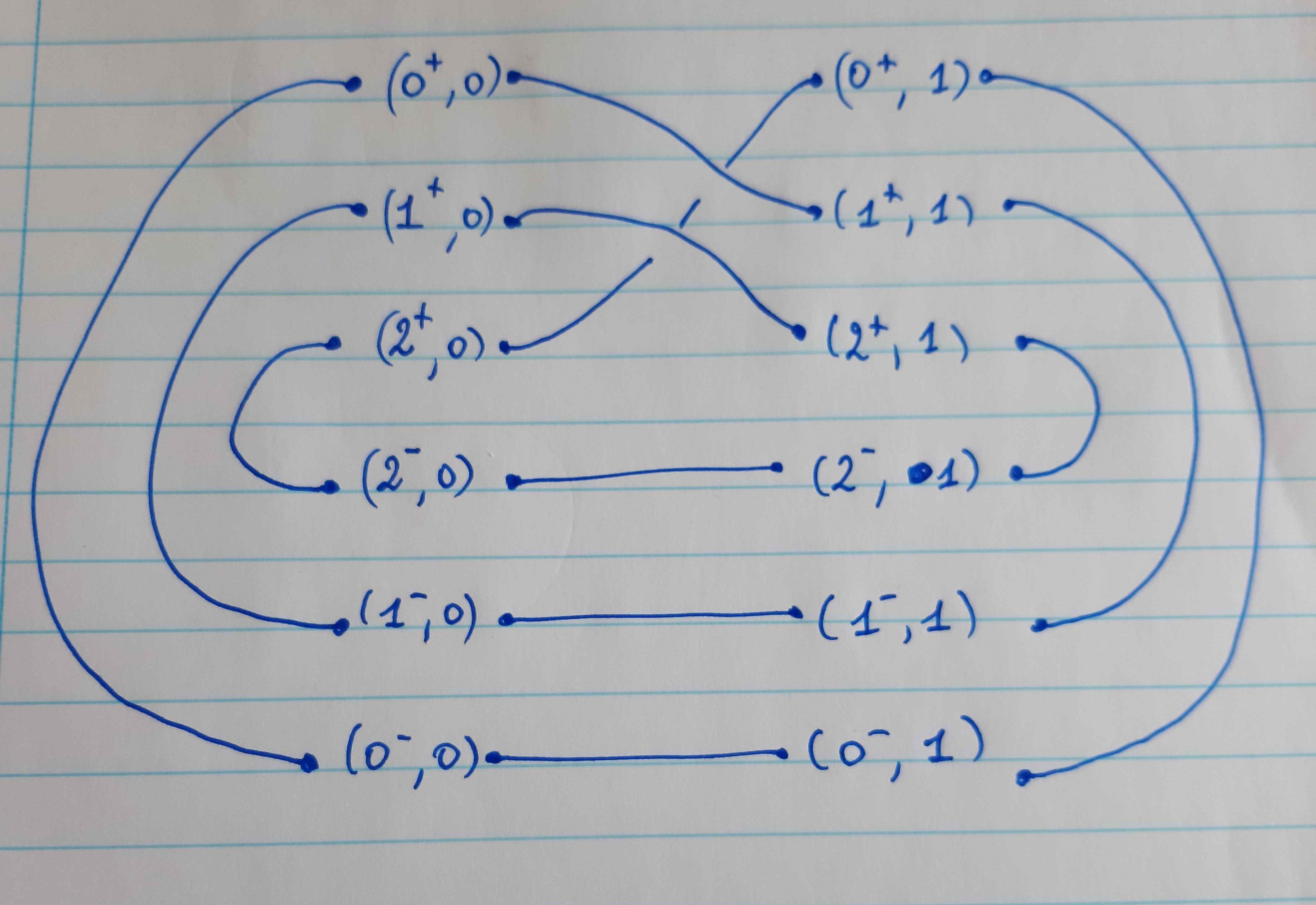}

We now need to identify the $C_n$-action on this circle. The proof that it is as we claim is not so difficult, but it takes up a bit of space. As it is mostly unrelated to the rest of the section, we defer it to \Cref{app:acttr}, specifically see \Cref{prop:actontr}. 
\end{proof}
\begin{cor}\label{cor:trpermnat}
    Let $B$ be a symmetric monoidal $(\infty,2)$-category. The functor $B^\dbl\to \End(\one_B)$ given by $b\mapsto \tr(b^{\otimes n};\sigma)$ is naturally equivalent to $$b\mapsto \dim(b)^{\otimes n(\sigma)}$$ 

    For $\sigma= C_n$ being the length $n$ cycle, the induced $C_n$-action is restricted from the $S^1$-action on $\dim(x)$. 
\end{cor}
\begin{proof}
    Using naturality of the $\Omega$-construction in \cite[(2.7)]{HSS}, we find that this functor is given by $$\Fun^{\otimes}_\colax(\Frrig(\pt),B)\to \Fun^{\otimes}_\colax(\Omega \Frrig(\pt),\Omega B) \xrightarrow{\mathrm{ev}}\Omega B$$
    where $\mathrm{ev}$ is evaluation at $\tr(x^{\otimes n};\sigma)\in \Omega\Frrig(\pt)$, where $x$ is the universal dualizable object. But now this trace has been computed in \Cref{lm:trperm}, and is indeed $\dim(x)^{\otimes n(\sigma)}$. The claim follows. 
\end{proof}
And second, we need a calculation of traces on \categories of local systems in $\PrL$:
\begin{lm}\label{lm:trspace}
    Let $X: A\to\Ss$ be a local system of spaces, and consider the induced local system $\Ss^X: A\to \PrL$. This has dualizable values, and the corresponding $\chi_X$ is given by the local system $LA\to \Ss$ whose unstraightening is $L(\colim_A X)\to LA$, that is, $$\gamma\mapsto  X^\gamma $$ where $X^\gamma$ is the fixed points for the $\Z$-action on $X$ induced by $\gamma$, or equivalently the global sections $\Gamma(S^1, \colim_A X \times_A S^1)$. 
\end{lm}
\begin{proof}
Let $\int X := \colim_A X$ for simplicity of notation, with its map $p:\int X\to A$. 

Note that by un/straightening, $X= p_! \pt$, where $\pt:\int X\to \Ss$ is the constant local system. Thus, in terms of functors $\PrL[A]\to \PrL$, letting $p^*$ denote the right adjoint to $p_!$, we find that the functor $\PrL[A]\to \PrL$ induced by $\Ss^X$ is equivalent to the composite $\PrL[A] \xrightarrow{p^*}~\PrL[\int X]\xrightarrow{r_!}~\PrL$,where $r:\int X\to \pt$ is the unique map.  

Thus, passing to $\PrL$-linear traces\footnote{The reader uncomfortable with set theory may pass to $\PrL_\kappa$ for $\kappa$ sufficiently large, this does not change anything to the argument.} and using (an obvious generalization of) \cite[Proposition 7.14]{CCRY}\footnote{The proof of this uses nothing specific to maps $A\to \pt$: it only uses \cite[Theorem 4.34]{CCRY} which is stated in full generality.} we obtain that $\chi_X$ is the composite: $$\Ss[LA]\xrightarrow{(Lp)^*}\Ss[L(\int X)]\xrightarrow{Lr}\Ss$$ 
This implies the desired claim under un/straightening. 
\end{proof}
With these two facts in mind, we can give: 
\begin{proof}[Proof of \Cref{prop:thhcolimweak}, \Cref{prop:thhcolimtext}]
    By \Cref{thm:CCRY}, we have 
    \begin{align*} 
   \THH((O\otimes C^{\otimes n})_{h\Sigma_n}\simeq \colim_{\sigma\in LB\Sigma_n}\THH(O\otimes C^{\otimes n};\sigma)\\
   \simeq \colim_{\sigma \in BC_n}\THH(O\otimes C^{\otimes n}; \sigma)\oplus \colim_{\sigma \in LB\Sigma_n \mid n(\sigma)>1} \THH(O\otimes C^{\otimes n},\sigma) \end{align*}

Now $\sigma$ acts diagonally on $O$ and $C^{\otimes n}$, so each of these terms splits as a colimit of  $\THH(O;\sigma)~\otimes~\THH(C^{\otimes n};\sigma)$. 

    The second summand is therefore a colimit of terms of the form $X\otimes \THH(C)^{\otimes k}$, $1<~k\leq~n$ by \Cref{cor:trpermnat} and is therefore indeed a sum of homogeneous functors of (varying) degrees $1<~k\leq~n$. 

    The first summand is, again by \Cref{cor:trpermnat}\footnote{Note that by \cite[Corollary 4.2.67]{Thesis}, the $S^1$-action on $\Tr(C)$, which is used in \Cref{cor:trpermnat}, agrees with that on $\THH(C)$.} and by \Cref{lm:trspace}, of the form $(O^\sigma~\otimes~\THH(C))_{hC_n}$, as was to be proved. 
\end{proof}
Putting things together, we obtain: 
\begin{cor}\label{cor:calculationTHH}
    Let $\Oo$ be a single-colored operad and $F: \Cat^\perf\rightleftharpoons \Alg_\Oo(\Cat^\perf):U$ the free-forgetful adjunction. There is a natural equivalence: $$\THH\circ UF\simeq \bigoplus_{n\geq 0}\bigoplus_{\sigma\in\Sigma_n/\mathrm{conj}}(\Oo(n)^\sigma \otimes \THH(C)^{\otimes n(\sigma)})_{hC(\sigma)}$$
\end{cor}

From this we can see how to approach the question of endomorphisms of $\THH\circ U$, and we can also see the relevant difficulties: we have already understood endomorphisms of $\THH$, and we see that we ``simply'' need to understand maps from $\THH$ to (say in the case of $\Oo=\mathrm{Comm}$) $\THH_{hC_n}$, $C_n\subset S^1$ a cyclic group, and maybe variants involving tensor powers. 

These variants are of higher polynomiality degree than $\THH$ itself, and so one might expect there to not be maps from $\THH$ to them, thus leading to a (potentially) simple expression. 

Three problems arise: 
\begin{enumerate}
    \item First, this is not quite right in general: there can a priori be maps from a homogeneous functor of degree $n$ to a homogeneous functor of degree $m>n$ - if the symmetric $m$-linear functor corresponding to the latter has an \emph{induced} symmetric structure, then this in fact does \emph{not} happen, but here the homotopy orbits by $C(\sigma)$ prevent that from being the case, except for very special operads $\Oo$. Another place where this does not happen is in chromatically localized contexts, as was first observed by Kuhn in \cite{kuhn};
    \item The second problem is that we have an infinite direct sum in the target of our mapping spectrum, and this can cause issues a priori both with the homogeneity phenomenon observed above, but also with computing the actual mapping spectrum once we've removed the ``high degree'' terms. This will turn out to be solved also in the chromatic context, because of a ``quick convergence'' phenomenon for Goodwillie derivatives in that context, a phenomenon recorded by Heuts in \cite{heuts} (though he attributes it to Mathew); 
    \item Finally, it turns out to be difficult to compute maps from $\THH$ to $\THH_{hC_n}$. Among other things, the reason for this is that $\THH_{hC_n}$ is simply not a module over $\THH$, and so the tools we have developped so far are not great for this purpose. The same ``quick convergence'' phenomenon mentioned above will turn out to save us here again in the chromatic context, though we end up also being able to compute this integrally. 
\end{enumerate}

These three problems, and their chromatically localized fixes are the reason our results in this section look the way they do, and are only partial. We point out that there \emph{are} in fact subtleties in the integral context, and that the answer we get in the chromatic world does not generalize verbatim to the integral context (thus, it is not only that our proof is not optimal, but that the answer is more complicated): 
\begin{obs}
    Let $p$ be a prime number. There is a nonzero natural transformation $\THH\to \Sigma \THH_{hC_p}$, given by the composite $$\THH\xrightarrow{\varphi_p}\THH^{tC_p}\to \Sigma \THH_{hC_p}$$ where the first map is the cyclotomic Frobenius, and the second is the attaching map for the fiber sequence defining the Tate construction. 

    This is nonzero e.g. because it is also the attaching map for the fiber sequence defining ($p$-typical) $\TR^2$, which in general does not split as $\THH\oplus\THH_{hC_p}$. 
\end{obs}
In fact, the partial splitting result from \Cref{thm:opns} will show that this nonzero transformation participates nonzero-ly to the endomorphism spectrum $$\eend_{\Fun(\CAlg(\Cat^\perf),\Sp)}(\THH\circ U,\THH\circ U)$$ in the case $\Oo=\mathrm{Comm}$, thus leading to genuinely new operations, not predicted e.g. by the chromatic case, or by the $\Z$-linear case explored in \cite{klamt}. In fact, we can fully compute each mapping spectrum $\map(\THH,\THH_{hC_n})$ and see that the above operations essentially ``generate'' it: 
\begin{thm}\label{thm:extraops}
    Let $n\geq 2$ be a prime number. The mapping spectrum $\map(\THH,\THH_{hC_n})$ is equivalent to $$\Sph[S^1/C_n]\oplus \bigoplus_{p\mid n, \textnormal{ prime }} \Omega\Sph_p[S^1/C_p]$$
    where the first factor is mapped into isomorphically by $\map(\THH,\THH)_{hC_n}$ along the assembly map, and where the degree $-1$ generators correspond to the Frobenius morphisms described above. 
\end{thm}
I wish to thank Achim Krause for an enlightening discussion that allowed me to prove this result. 

Before we get into the proof, we need a computation of $\THH(A_k)$ which is suitably natural. We let $T:\Sp\to \Alg(\Sp)$ denote the left adjoint to the forgetful functor, i.e. the free (associative) algebra functor\footnote{$T$ stands for ``tensor algebra''. }. This is a canonically augmented algebra, since $T(0)=\Sph$, and for a functor $F: \Alg(\Sp)\to \E$ with values in some stable \category, we let $\tilde F(T(M))$ denote the complementary summand to $F(\Sph) = F(T(0))$. 
\begin{prop}\label{prop:thhsqz}
    There is a natural equivalence $$\widetilde \THH(T(M))\simeq \bigoplus_{d\geq 1} \Ind^{S^1}_{C_d} M^{\otimes d}$$
    where: $M^{\otimes d}$ has a $C_d$-action by permutation, and $\Ind_{C_d}^{S^1}$ is left adjoint to the restriction functor $\Sp^{BS^1}\to \Sp^{BC_d}$. 
\end{prop}
\begin{proof}
    This follows from combining \cite[Theorem D]{LT2} and \cite[Proposition 3.2.2, Proposition 4.5.1]{raskin}\footnote{Note that Raskin uses $\Ind$ to denote the \emph{right} adjoint to the forgetful functor. However, they differ by a shift $\Sigma$, and this is accounted for by the motivic pullback square from \cite{LT2}} and noting that $\THH(\Sph,N)\simeq N$ naturally in $N$.
\end{proof}

\begin{proof}[Proof of \Cref{thm:extraops}]
    We first note that, as a spectrum with $S^1$-action coming from the target, $\map(\THH,\THH)\simeq\Sph[S^1]$ has an induced $S^1$-action, so the norm map $\map(\THH,\THH)_{hC_n}\to \map(\THH,\THH)^{hC_n}$ is an equivalence. 
    
    Since it factors through (via assembly and coassembly) the norm map $\map(\THH,\THH_{hC_n})\to \map(\THH,\THH^{hC_n})\simeq \map(\THH,\THH)^{hC_n}$, it follows that $\map(\THH,\THH_{hC_n})$ admits $\map(\THH,\THH)_{hC_n}$ as a summand.

    We now move on to the explicit calculation of the extra summand. By \Cref{cor:coder}, there is an equivalence $$\map(\THH,\THH_{hC_n}) \simeq \lim_k\Sigma^k(\THH(\tilde A_k)_{hC_n})$$ and recall from \Cref{prop:thhsqz} that $$\THH(\tilde A_k)\simeq \bigoplus_{d\geq 1} \Ind^{S^1}_{C_d}(\Omega^k \Sph)^{\otimes d})$$

  \begin{rmk}\label{rmk:withoutCn} We make a digression to mention that without $_{hC_n}$, this inverse limit is easy to deal with: the map $\Sigma^{k+1} (\Omega^{k+1} \Sph)^{\otimes d}\to \Sigma^k (\Omega^k \Sph)^{\otimes d}$ is null whenever $d\geq 2$, and thus the whole inverse system is $\Z$-equivariantly nilpotent of order $2$, where $\Z$ acts via picking a generator $\Z\to C_d$. Since the underlying object of $\Ind^{S^1}_{C_d}X$ is $X_{h\Z}$ for $X\in\Sp^{BC_d}$, it follows from \Cref{lm:sumnilp} that this inverse limit is equivalent to the inverse limit of the $d=1$ summand, which is a constant diagram. This is another way of proving that $\map(\THH,\THH)\simeq \Sph[S^1]$, but will also be relevant in a second. 

    The difficulty now is that for $\THH_{hC_n}$, we have to take into account the $C_n$-equivariance, and the map $\Sigma^{k+1} (\Omega^{k+1} \Sph)^{\otimes d}\to \Sigma^k (\Omega^k \Sph)^{\otimes d}$ is not $C_d$-equivariantly null in general. 
\end{rmk}
    Now, using a basechange formula, we obtain that for $d,k\geq 1$, $$(\Ind^{S^1}_{C_d} X)_{hC_n} \simeq X_{h(C_{d\wedge n} \times \Z)}$$ where the $\Z\times C_{d\wedge n}$-action on $X$ is induced by the map $\Z\times C_{d\wedge n}\to C_d$ which is the inclusion on the right factor, and picks out a generator on the left factor - we use $a\wedge b$ to denote the gcd of $a$ and $b$. 

    Let us point out that $\Z$-orbits are given by a colimit over $S^1$, i.e. a finite colimit, and therefore commute with the limit term. 

    So all in all we are considering the limit of $(\bigoplus_{d\geq 1}\Sigma^k (\Sph^{\otimes dk})_{hC_{d\wedge n}} )_{h\Z}$. It is not hard to check that the maps $\Sigma^{k+1}((\Sph^{\otimes d(k+1)})_{hC_{d\wedge n}})\to \Sigma^{k}((\Sph^{\otimes dk})_{hC_{d\wedge n}})$ are $C_{d\wedge n}$-orbits of Euler maps for the standard representation $\mathbb R^d$ of $C_{d\wedge n}$. 

    If $d\wedge n < d$, then $\mathbb R^d$ admits a more than $1$-dimensional fixed point space for $C_{d\wedge n}$, and thus this Euler map is $C_{d\wedge n}$-equivariantly nullhomotopic. It follows that the limit is the same as the limit taken over the $d$'s for which $d\wedge n = d$, i.e. $d\mid n$.

Thus we are dealing with finite sums and can stop worrying about commuting limits and finite sums and focus on each $$\lim_k \Sigma^k ((\Omega^k\Sph)^{\otimes d})_{hC_d}, d\mid n$$ (again, the $\Z$-orbits commute with the inverse limit, so we will deal with them later).

We claim that these are $0$ whenever $d$ is not $1$ or prime, and $\Omega \Sph_p$ whenever $d=p$ is prime. The term corresponding to $d=1$ comes from $\map(\THH,\THH)_{hC_n}$, by \Cref{rmk:withoutCn}, and therefore we can also focus on the case where $d$ is not $1$. Once this is proved, we will be done: the only nonzero terms correspond to the primes dividing $n$, they are given by $(\Omega\Sph_p)_{h\Z} \simeq \Omega\Sph_p[S^1/C_p]$\footnote{The $/C_p$ does not affect the homotopy type, but is here to remember the specific $S^1$-action. The fact that it is given exactly as that one can be recovered by remembering where this $(-)_{h\Z}$ came from. Since we will not specifically need this $S^1$-action, we do not give more details.}, as claimed.

Let us prove the second fact, i.e. the case of $d=p$ being prime: in this case, the map $$\Sigma^k ((\Omega^k\Sph)^{\otimes p})_{hC_p}\to \Sigma^k ((\Omega^k\Sph)^{\otimes p})^{hC_p}$$ has $\Sigma^k((\Omega^k \Sph)^{\otimes p})^{tC_p}$ as a cofiber. Since $X\mapsto (X^{\otimes p})^{tC_p}$ is an exact functor by \cite[Proposition III.1.1]{NS}, we find that the cofiber, as a $\NN\op$-indexed diagram, is equivalent to the constant diagram at $\Sph^{tC_p}$, which is equivalent to $\Sph_p$ by the Segal conjecture, a theorem of Gunawardena \cite{gunawardenasegal} (at odd primes) and Lin \cite{linsegal} (at $p=2$).

Since $\lim_k \Sigma^k ((\Omega^k \Sph)^{\otimes p})^{hC_p} \simeq (\lim_k (\Omega^k \Sph)^{\otimes p})^{hC_p} = 0$, the result follows. The latter equality comes from the fact that the maps in the system are non-equivariantly trivial, and hence the limit is trivial before taking homotopy fixed points, therefore it is also trivial after taking homotopy fixed points. 

Our proof of the first fact, namely the vanishing of the inverse limit for $d$ composite proceeds differently depending on whether $d$ is a prime power or has different prime factors.

The key input will be the following easy observation: for any proper subgroup $C_a \subset C_d$, $\mathbb R^d$, viewed as a $C_a$-representation, has a fixed point space of dimension $>1$ and thus the associated Euler map is $C_a$-equivariantly nullhomotopic. 

\textbf{The case of several prime factors:} We consider more generally the Euler map for a spectrum $M$: $$\Sigma (\Omega M)^{\otimes d}\to M^{\otimes d}$$

So let $a,b$ be coprime integers $>1$ such that $ab=d$. Using the associated fracture square, we may assume that either $a$ or $b$ is invertible. If $b$ is invertible (say, the argument is symmetric), then we use the above fact to obtain that the induced map $$(\Sigma (\Omega M)^{\otimes d})_{hC_a}\to (M^{\otimes d})_{hC_a}$$ is null, and $b$ acts invertibly on both terms. Thus, it is also $C_b$-equivariantly null, and so in total, the map $$(\Sigma (\Omega M)^{\otimes d})_{hC_d}\to (M^{\otimes d})_{hC_d}$$ is null, which implies that the inverse limit is $0$.

\textbf{The case of a prime power:} 
Say $d= p^s, s>1$. In this case, we use the Tate orbit lemma \cite[Lemma I.2.1]{NS}. An immediate corollary of this is that for $X$ bounded below with a $C_{p^s}$-action, we have $(X_{hC_{p^{s-1}}})^{tC_{p^s}/C_{p^{s-1}}}=0$. A way to rephrase this is that the norm map $$X_{hC_{p^s}}=(X_{hC_{p^{s-1}}})_{hC_{p^s}/C_{p^{s-1}}} \to (X_{hC_{p^{s-1}}})^{hC_{p^s}/C_{p^{s-1}}}$$ is an equivalence. 

In particular, we obtain $$(\lim_k \Sigma^k (\Omega^k\Sph)^{\otimes p^s})_{hC_{p^s}} = \lim_k \Sigma^k (((\Omega^k\Sph)^{\otimes p^s})_{hC_{p^{s-1}}})^{hC_{p^s}/C_{p^{s-1}}} = (\lim_k \Sigma^k ((\Omega^k\Sph)^{\otimes p^s})_{hC_{p^{s-1}}})^{hC_{p^s}/C_{p^{s-1}}} $$
and it suffices to prove that the inner term is $0$. But now $C_{p^{s-1}}$ is a proper subgroup of $C_{p^s}$ so the Euler maps are again $C_{p^{s-1}}$-equivariantly null, and thus the inner term is indeed $0$.
\end{proof}
\begin{rmk}
    Let us point out that in the last part of this proof, the composite case and the prime power case have qualitatively different proofs: in the composite case, we essentially obtain that the relevant inverse system is nilpotent, because nilpotent inverse systems are closed under pullback; whereas in the prime power case we only get that the inverse limit is $0$. It seems likely that the inverse system is actually not pro-zero.
\end{rmk}

We can also use this result to study the independently interesting question of the uniqueness of cyclotomic Frobenii on $\THH$. Indeed, maps from $\THH$ to $\THH^{hC_p}$ are easily computable. We obtain: 
\begin{cor}
    The spectrum $\map(\THH,\THH^{tC_p})$ is equivalent to $\Sph_p[S^1/C_p]$, and in particular the spectrum of $S^1$-equivariant maps $\THH\to\THH^{tC_p}$ is equivalent to $\Sph_p$. 
\end{cor}
\begin{proof}
    This follows from \Cref{thm:extraops} and the fact that $$\map(\THH,\THH^{hC_p})\simeq \map(\THH,\THH)^{hC_p}$$ 
\end{proof}
Similarly, and using the same methods as for the proof of the multiplicative part of \Cref{thm:endplain}, we obtain: 
\begin{cor}
    The space of lax symmetric monoidal transformations $\Map^\otimes(\THH,\THH^{tC_p})$ is equivalent to $\Map_{\CAlg(\Sp)}(\Sph^{S^1},\Sph_p)$, which is in turn equivalent to $\Map(S^1,S^1_p)$, where $S^1_p$ is the $p$-complete circle. 
\end{cor}

Before moving on to proofs of the general results about $\Oo$-algebras, we point out one final thing: our methods, both for the partial result in the integral case, and the full result in the chromatically localized context, extend to ``computing'' (with the same caveats) mapping spectra between (pointwise) tensor powers of $\THH$. However, the expression of our answer there is more complicated, and still only partial, thus we refrain from stating and proving it, as the methods are exactly the same. 

We now move on to the proofs. We start with the integral version, which we can prove right away with no further preparation:
\begin{proof}[Proof of (i) in \Cref{thm:opns}]
   The map we construct is the following : for each $n$, we pick the inclusion of the length $n$-cycle $\gamma_n$ in $\Sigma_n$ and produce the following composite: \begin{align*}
       \Sph[\Oo(n)^{\gamma_n}\times S^1]_{hC_n} \simeq \map(\THH, \Oo(n)^{\gamma_n}\otimes \THH)_{hC_n}\\
       \to \map(\THH, (\Oo(n)^{\gamma_n}\otimes\THH)_{hC_n}) \to \map(\THH,\THH\circ UF)\\
       \simeq \map(\THH\circ U,\THH\circ U) \end{align*}
       where the map $(\Oo(n)^{\gamma_n}\otimes\THH)_{hC_n}\to \THH\circ UF$ is the summand inclusion from \Cref{prop:thhcolimtext}. 

   To prove that each finite direct sum of these inclusion splits, we give the argument for each summand for simplicity of writing, but it clearly extends to the general case, e.g. by using norm maps for non-connected groupoids. 
   
   The point is that the last maps involved in this composite split because of the computation of $\THH\circ UF$ from \Cref{prop:thhcolimtext}, while the assembly map from the $C_n$-orbits outside to the $C_n$-orbits inside is split by the norm map (it is the same argument as in the beginning of the proof of \Cref{thm:extraops}: if $F$ is a functor that preserves limits, and $F(x)_{hC_n}\to F(x)^{hC_n}$ is an equivalence, then the assembly map $F(x)_{hC_n}\to F(x_{hC_n})$ splits). 
\end{proof}

We now move on to the chromatically localized picture. The key piece of intuition here is the following result, which is a corollary of (a suitably generalized version of) Kuhn's \cite[Theorem 1.1]{kuhn}:
\begin{cor}
    Let $\E$ be a cocomplete $T(n)$-local stable \category, where $n\geq 0$, and $C$ be stable \category. For any $m_0<m_1$, any $m_i$-homogeneous functors $F_i: C\to\E$, $\Map(F_0,F_1)=0$. 
\end{cor}
\begin{rmk}
    In the other direction, note that the definition of $m_1$-homogeneous includes $(m_1-1)$- and thus $m_0$-reduced, so that $\Map(F_1,F_0) =0$ is clear. We note that this direction is simply not true without $T(n)$-localization, as the existence of $k$-excisive functors that do not split as sums of homogeneous ones shows (e.g. norm functors in equivariant homotopy theory, or more classically derived symmetric/exterior powers functors). 
\end{rmk}
Kuhn's proof is more direct\footnote{The way to deduce this from Kuhn's result is to take a map $F_0\to F_1$ and consider its cofiber. Kuhn's splitting result implies that it splits off $\Sigma F_0$ and thus the map $F_0\to F_1$ must have been $0$.}, but we give a proof in terms of \Cref{prop:heuts} which we will need later anyways. 
\begin{proof}
By \Cref{prop:heuts}, for any finite type $n$ spectrum $V$, $V\otimes P^{m_0}F_1=0$, where $P^{m_0}$ denote the $m_0$th co-Goodwillie derivative. 

Since $\E$ is $T(n)$-local, it follows that $P^{m_0}F_1=0$, and thus $$\Map(F_0,F_1) = \Map(F_0,P^{m_0}F_1)= 0$$
\end{proof}
As we mentioned before, this corollary will in fact not quite cut it because of the presence of infinite direct sums in our considerations. 

The point is that (as ``our'' proof shows) we can reinterpret this statement as a vanishing statement for the $m_0$-th \emph{co}-Goodwillie derivative\footnote{Or Goodwillie coderivative ?} of $F_1$, something which involves infinite inverse limits (as Goodwillie derivatives involve infinite colimits), which thus do not commute \emph{a priori} with the relevant colimits. The idea, and the point where $T(n)$-localization will again be relevant, is that in fact the vanishing of this co-Goodwillie derivative happens \emph{quickly}: the relevant inverse limits, rather than simply being $0$, are \emph{pro}-zero, which allows us to get the corresponding result for infinite colimits. 

Let us briefly explain the idea that will allow us to deal with these infinite colimits, in the (much) simpler special case of $m_0 = 1$, so that $F_0$ is an exact functor, and $n=0$ so that $T(n)$-local means rational.

In that case, the first co-Goodwillie derivative of $F_1$ is given by the inverse limit $\lim_n \Sigma^n F_1(\Omega^n -)$. Now, $F_1$ being $m_1$-homogeneous is of the form $(f_1\circ \delta_n)_{h\Sigma_n}$ for some symmetric $n$-linear functor $f_1$ \cite[Proposition 6.1.4.14]{HA}, and we note that each of the transition maps $\Sigma f_1(\Omega X_1,...,\Omega X_{m_1}) \to f_1(X_1,...,X_{m_1})$ is nullhomotopic, because they are (dual to) the diagonal map $S^1\to (S^1)^{\wedge n} = S^n$. However, they are \emph{not} $\Sigma_n$-equivariantly null in general. Rationally, though, they are because null implies equivariantly null for a finite group over a rational base. 

Thus the relevant inverse system is even more than pro-zero: all of the transition maps are $0$. 

Moving to a higher $m_0$ and possibly a higher height, this is no longer true, but the pro-zero-ness persists, although the argument is more involved. This is explicitly recorded in \cite[Appendix B]{heuts}, where Heuts attributes it to Akhil Mathew. For this statement to make sense, recall from \cite[Construction 6.1.1.27]{HA} that $n$-excisive approximations are given as certain explicit sequential colimits of finite limits - thus, dually, the $n$-excisive coapproximation $P^nF$ of a functor $F$ is given as a sequential limit of a certain inverse system. In the following statement, it is this inverse system we call ``the inverse system computing $P^nF$'': 
\begin{prop}[{\cite[Appendix B]{heuts}}]\label{prop:heuts}
    Let $n\geq 0$, $V$ a finite type $n$ spectrum and let $k\geq 0$ be an integer. There is a constant $C$ such that for any $m>k$ and any $m$-homogeneous functor $F: D\to \E$ from a stable \category $D$ to a cocomplete stable \category $\E$, the inverse system computing $V\otimes P^k F$, the $k$-th co-Goodwillie derivative is nilpotent of exponent $C$.
\end{prop}
Here, we used: 
\begin{defn}
    Let $I$ be a cofiltered \category, and $A_\bullet: I\to \E$ be a diagram with values in a pointed \category. $A_\bullet$ is called nilpotent of exponent $N$ if for every sequence of $N$ composable non-identity morphisms $i_0\to ... \to i_N$, $A_{i_0}\to A_{i_N}$ is nullhomotopic. 
\end{defn}
If $I$ has an initial copy of $\mathbb N\op$, then any nilpotent $I$-shaped diagram has a trivial limit (in fact is pro-zero, but being nilpotent is stronger than being pro-zero). 

But nilpotency with a uniform exponent allows for more, as the following easy observation shows:
\begin{lm}\label{lm:sumnilp}
    Let $J$ be any set, $I$ a cofiltered \category and $X_j, j\in J$ a family of $I$-shaped diagrams that are nilpotent of a uniform exponent $C$. The diagram $\bigoplus_J X_j$ is also nilpotent of degree $C$. In particular, if $I$ has an initial copy of $\mathbb N\op$, then $\lim_I \bigoplus_J X_j  = 0$.
    \end{lm}
    \begin{proof}
This is clear, as direct sums of zero morphisms are $0$. 
    \end{proof}
We are now ready to prove the second part of \Cref{thm:opns}: 
\begin{proof}[Proof of (ii) in \Cref{thm:opns}]
    For the duration of this proof, everything is implicitly rationalized, or $T(n)$-localized for some implicit prime and some height $n\geq 1$. 

    We recall that $$\map_{\Fun(\Alg_\Oo(\Cat^\perf),\Sp)}(\THH \circ U, \THH \circ U)\simeq \map_{\Fun(\Cat^\perf,\Sp)}(\THH, \THH \circ UF)$$

   We combine \Cref{prop:heuts}, \Cref{lm:sumnilp},\Cref{cor:calculationTHH} and \Cref{thm:motlocisloc} to find that this is equivalent to  $$\map_{\Fun(\Mot_\loc,\Sp)}(\THH,\bigoplus_{n\geq 1}\bigoplus_{\sigma \in\Sigma_n/\mathrm{conj} \mid n(\sigma) = 1} (\Oo(n)^\sigma\otimes \THH^{\otimes n(\sigma)})_{hC(\sigma)})$$

Now, for every $n$, there is only one permutation $C_n \in \Sigma_n$ with $n(\sigma) =1$, up to conjugacy: the length $n$ cycle. Its centralizer is itself the length $n$ cycle, and by \Cref{cor:trpermnat}\footnote{Though the proof really is in \Cref{prop:actontr}.}, it acts on $\THH$ via the canonical map $C_n\to S^1$. Thus, we get an equivalence with $$\map(\THH, \bigoplus_{n\geq 1} (\Oo(n)^{\gamma_n}\otimes\THH)_{hC_n})$$

Now, we use \Cref{cor:coder} to rewrite this as $$\lim_k (\bigoplus_{n\geq 1}\Sigma^k (\Oo(n)^{\gamma_n}\otimes \widetilde\THH(A_k)_{hC_n}) $$

Now we use again \Cref{prop:heuts}, \Cref{lm:sumnilp} and \Cref{prop:thhsqz} to obtain $$\lim_k \bigoplus_{n\geq 1} \Sigma^k (\Oo(n)^{\gamma_n}\otimes \Ind_e^{S^1} \Omega^k\Sph)_{hC_n}\simeq \bigoplus_{n\geq 1} (\Oo(n)^{\gamma_n}\otimes \Sph[S^1])_{hC_n} \simeq \bigoplus_{ n\geq 1}\Sph[\Oo(n)^{\gamma_n}\times S^1]_{hC_n}$$ which was to be proved. 
 \end{proof}
As explained in the beginning of this section, this proof does not work integrally: it should be clear that we used \Cref{prop:heuts} repeatedly, and this is crucially a truly chromatic or rational phenomenon. We have, in any case, indicated in \Cref{thm:extraops} that the result fails in the integral case. I do not have a guess for what the correct answer is: it is not clear to me whether the operations from \Cref{thm:extraops} account for all the difference in the integral case, or if the commutation between inverse limits and infinite direct sums fails badly enough that there are other ``exotic'' operations. 
\appendix

\section{Splitting motives}\label{ch:motsplit}
In this appendix, we gather a few facts about splitting motives\footnote{These are typically called ``additive motives'' in the literature. This, especially the corresponding notion of ``additive invariant'' sounds very confusing and less descriptive than ``splitting'', so I have opted for this change of name.} which will be convenient for \Cref{section:day}. 

First, let us fix some notation and terminology. 
\begin{nota}
    Let $C$ be a stable \category. We let $\Cof(C)\subset C^\square$ denote the full subcategory of cofiber sequences, that is, squares whose bottom left corner is $0$, and that are coCartesian. 
\end{nota}
This is the analogue of what Waldhausen called $S_2(C)$ in the context of his $S_\bullet$-construction. 
\begin{defn}
    A split exact sequence is a sequence $C\xrightarrow{i}D\xrightarrow{p}Q$ of stable \categories where: 
    \begin{enumerate}
        \item $p\circ i= 0$;
        \item $p$ and $i$ both have right adjoints $p^R$ and $i^R$ respectively;
        \item $i$ and $p^R$ are fully faithful\footnote{Equivalently, the unit $\id_C\to i^Ri$ and the counit $pp^R\to\id_Q$ are equivalences.};
        \item the canonical nullsequence\footnote{For any functors $f,g$, $\map(if,p^Rg)= 0$, so this nullsequence is canonical and unique.} $ii^R\to \id_D\to p^Rp$ is a cofiber sequence.
    \end{enumerate} 
\end{defn}
\begin{rmk}
Split exact sequences are essentially equivalent to semi-orthogonal decompositions. 
\end{rmk}
The following is the key example of a split exact sequence, it is essentially universal:
\begin{ex}
    Let $C$ be a stable \category, and let $i: C\to \Cof(C)$ be given by $x\mapsto~(x\to~x\to~0)$, $p: \Cof(C)\to C$ be given by evaluation the the bottom right corner, so $(x\to y\to z)\mapsto z$.

    $i^R$ is given by evaluation at the top left corner, while $p^R: z\mapsto (0\to z\to z)$.

    We let $f, t,c: \Cof(C)\to C$ denote the evaluation functors at the three nonzero corners\footnote{$f$ stands for ``fiber'', $t$ for ``total'' and $c$ for ``cofiber''.}. They fit into a canonical cofiber sequence $f\to t\to c$ of functors $\Cof(C)~\to~C$. 
\end{ex}
We further recall that $\Cat^\ext$ is a semi-additive \category, so that for any \category $\E$ with finite products, every finite-product preserving functor $f~:~\Cat^\ext~\to~\E$ admits a canonical lift to $\CMon(\E)$, so we can add maps between values of $f$. In this context, we have: 
\begin{thm}[Waldhausen]\label{appthm:wald}
Let $E:\Cat^\ext\to \E$ be a finite-product-preserving functor. The following are equivalent: 
\begin{enumerate}
    \item For any cofiber sequence of exact functors $F,G,H : C\to D$, $F\to G\to H$, there exists some homotopy $E(G)\simeq E(F)+E(H)$; 
    \item For any stable $\infty$-category $C$, and for the canonical cofiber sequence of functors $f\to~t\to~c : \Cof(C)\to C$, $E(t)\simeq E(f)+E(c)$; 
    \item For any stable $C$, the functor $E$ applied to $\Cof(C)\xrightarrow{(f,c)} C\times C$ yields an equivalence; 
    \item For any split exact sequence $C\xrightarrow{i}D\xrightarrow{p}Q$, letting $r$ denote the right adjoint to $i$, $E$ applied to $D\xrightarrow{(r,p)}C\times Q$ yields an equivalence. 
\end{enumerate}
    
\end{thm}
\begin{defn}
    A functor $E:\Cat^\ext\to \E$ is called a splitting invariant if it preserves finite products and satisfies any (and hence all) of the equivalent conditions of the previous theorem. 

    It is called \emph{finitary} if it preserves filtered colimits. 
\end{defn}
\begin{defn}
    Let $\mathcal U_\add: \Cat^\ext\to \Mot_\add$ denote the universal finitary splitting invariant with values in a stable \category\footnote{Strictly speaking, we do not need stability here. It turns out to be more convenient for our use of $\Mot_\add$ in \Cref{section:day} - though not necessary - so we choose this convention, but for other purposes it is maybe more canonical not to require this, though the not-necessarily-stable version automatically embeds in the stable version.}. It exists by \cite{BGT13}, and $\Mot_\add$ admits a unique presentably symmetric monoidal structure for which $\mathcal U_\add$ is symmetric monoidal.
\end{defn}
We take for granted\footnote{Or for definition. } the following main result from \cite{BGT13}, and deduce two keys fact about mapping spaces in $\Mot_\add$ from it which  we will use in \Cref{section:day}.
\begin{thm}[{\cite[Theorem 1.3]{BGT13}}]
    There is an equivalence $$K^\cn \simeq \map(\mathcal U_\add(\Sp^\omega), \mathcal U_\add(-))$$
\end{thm}
Here, $K^\cn$ denotes connective $K$-theory. 
\begin{rmk}
    We depart slightly from \cite{BGT13} by not requiring splitting invariants to be invariant under idempotent-completion. This is a mild difference, and only affects $K_0$.
\end{rmk}
We can essentially take this theorem as our definition of $K^\cn$. Thus, the following is only really a corollary if one has an \emph{a priori} definition of $K$-theory.

Before we state it, note that we have a canonical internal hom comparison map of the form: 
$$\mathcal U_\add(\Fun^\ext(A,B))\to \hom(\mathcal U_\add(A),\mathcal U_\add(B))$$
whenever $A,B\in\Cat^\ext$. By taking maps from the unit, this induces an map $$K^\cn(\Fun^\ext(A,B))\to \map_{\Mot_\add}(\mathcal U_\add(A),\mathcal U_\add(B))$$

\begin{cor}\label{cor:mapinmot}
    Let $A,B\in\Cat^\ext$, with $A$ being a compact object therein. The canonical map
    $$K^\cn(\Fun^\ext(A,B))\to \map_{\Mot_\add}(\mathcal U_\add(A),\mathcal U_\add(B))$$ 
    described above is an equivalence. 
\end{cor}
\begin{proof}
    Fix $A$, and consider $B$ a variable. 

    Since $A$ is compact, $K^\cn(\Fun^\ext(A,-))$ is finitary, and since $\Fun^\ext(A,-)$ preserves split exact sequences\footnote{This is the crucial difference with the story in the localizing case.}, $K^\cn(\Fun^\ext(A,-))$ is a splitting invariant, hence in total a finitary splitting invariant. 

    One also proves ahead of time (essentially using item (iii) in \Cref{appthm:wald}) that $\mathcal U_\add(A)$ is compact in $\Mot_\add$, so that $\map_{\Mot_\add}(\mathcal U_\add(A),\mathcal U_\add(-))$ is also a finitary splitting invariant. 
    
    We now use the Yoneda lemma: let $F: \Cat^\ext\to \Sp$ be a finitary splitting invariant. We then have equivalences: $$\map(K^\cn(\Fun^\ext(A,-)),F)\simeq \map(K^\cn, F(A\otimes -)) \simeq F(A\otimes\Sp^\omega)\simeq F(A)$$

    where the first equivalence is by adjunction, the second uses the Yoneda lemma with $K^\cn \simeq \map_{\Mot_\add}(\mathcal U_\add(\Sp^\omega),\mathcal U_\add(-))$ and the fact that $F(A\otimes -)$ is still a finitary splitting invariant. 

    By the Yoneda lemma, we also have $$\map(\map_{\Mot_\add}(\mathcal U_\add(A),\mathcal U_\add(-)),F)\simeq F(A)$$

    It follows that there is an equivalence as desired. It is not difficult to verify that the equivalence in question is indeed given by the internal hom-comparison map, for example by rewriting this second equivalence as the string: 
    $$\map(\map_{\Mot_\add}(\mathcal U_\add(A),\mathcal U_\add(-)),F)\simeq \map(\map(\mathcal U_\add(\Sp^\omega), \hom(\mathcal U_\add(A),\mathcal U_\add(-))), F)$$ 
    $$\simeq \map(K^\cn, F(\mathcal U_\add(A)\otimes -))\simeq F(A)$$
\end{proof}

\begin{cor}\label{cor:internalhominmot}
    Let $A,B\in \Cat^\ext$, with $A$ being a compact object therein. The internal hom-comparison map $$\mathcal U_\add(\Fun^\ext(A,B))\to \hom(\mathcal U_\add(A),\mathcal U_\add(B))$$ 
    is an equivalence.
\end{cor}
\begin{proof}
    It suffices to check that mapping in from $\mathcal U_\add(C), C\in(\Cat^\ext)^\omega$ produces an equivalence. 

    Since $ C$ is compact, \Cref{cor:mapinmot} shows that mapping in from $C$ gives the following map: 
    $$K^\cn(\Fun^\ext(C,\Fun^\ext(A,B)))\to \map_{\Mot_\add}(\mathcal U_\add(C), \hom(\mathcal U_\add(A),\mathcal U_\add(B)))$$

    which is, in turn: 
    $$K^\cn(\Fun^\ext(C\otimes A,B))\to \map_{\Mot_\add}(\mathcal U_\add(C)\otimes\mathcal U_\add(A),\mathcal U_\add(B))$$
    Since $\mathcal U_\add$ is strong monoidal, using that $A\otimes C$ is compact and using again \Cref{cor:mapinmot}, we conclude that this map is an equivalence. 
\end{proof}

Beyond splitting invariants/motives, there is the crucial notion of a \emph{localizing} invariant/motive\footnote{Our localizing invariants will be Morita invariant, also known as Karoubi localizing invariants. In this setting, this is the standard convention, though we note that $K^\cn$ is an interesting example of a ``Verdier localizing invariant'' which is not Morita invariant.}. 
\begin{defn}
    A null-sequence $C\to D\to Q$ of stable \categories is called a Karoubi sequence, or localization sequence, if $\Ind(C)\to \Ind(D)\to \Ind(Q)$ is a split exact sequence. 

    One can show that this is equivalent to: 
    \begin{enumerate}
        \item $C\to D$ is fully faithful; 
        \item the sequence is a cofiber sequence in $\Cat^\perf$, the \category of idempotent-complete stable \categories. 
    \end{enumerate}
\end{defn}
With this in hand, we can define: 
\begin{defn}
    A localizing invariant with values in a stable \category $\E$ is a functor $E:\Cat^\perf\to \E$ sending $0$ to $0$ and Karoubi sequences to co/fiber sequences. 
    
It is called finitary if it preserves filtered colimits. 
\end{defn}
And finally, we have: 
\begin{defn}
    We let $\mathcal U_\loc : \Cat^\perf\to \Mot_\loc$ denote the universal finitary localizing invariant. It exists by \cite{BGT13}, and $\Mot_\loc$ is presentable and stable. 
\end{defn}

A result which we do not prove here, but will be used once in the body of the paper is the following result, obtained in joint work with Sosnilo and Winges \cite{RSW}: 
\begin{thm}\label{thm:motlocisloc}
    The functor $\mathcal U_\loc: \Cat^\perf\to \Mot_\loc$ is a Dwyer-Kan localization (at the class of morphisms it inverts, called motivic equivalences). 
\end{thm}
The point of this theorem is that it allows us to compare mapping spaces/spectra of the form $\Map_{\Fun(\Mot_\loc,\E)}(F,G)$ and $\Map_{\Fun(\Cat^\perf,\E)}(F\circ \mathcal U_\loc,G\circ \mathcal U_\loc)$ even when $F,G$ need not be cocontinuous functors out of $\Mot_\loc$. We will use this for $F,G$ being functors such as $\THH^{\otimes n}$, which, while clearly not cocontinuous in general (it is $n$-excisive), still clearly factors through $\Mot_\loc$.

\section{Actions on traces}\label{app:acttr}
Let $\C$ be a symmetric monoidal \category. In \cite{HSS}, the authors construct a trace map $$\tr: \Map(S^1,\C^\dbl)\to \End(\one_\C)$$ 
It is $S^1$-equivariant, and in particular induces a map $\C^\dbl\to \End(\one_\C)^{BS^1}$ upon taking $S^1$-fixed points. 

In particular, whenever $A\to\C^\dbl$ is a map from a space $A$, we obtain a map $LA\to~\End(\one_\C)$. 

\begin{ex}\label{ex:actionontr}
    Let $x\in\C^\dbl$, and consider the $\Sigma_n$-action on $x^{\otimes n}$. It corresponds formally to a map $B\Sigma_n\to \C^\dbl$ which in turn induces a map $LB\Sigma_n\to \map(S^1,\C^\dbl)\to \End(\one_\C)$. 

    Restricting, e.g. to the conjugacy class of the length $n$-cycle $\sigma$ in $\Sigma_n$, we obtain a map $BC_n\to \End(\one_\C)$ sending the point to $\tr(x^{\otimes n};\sigma)$, i.e. a $C_n$-action on this trace. 
\end{ex}

 We have observed in \Cref{lm:trperm} that for this length $n$ cycle, $\tr(x^{\otimes n};\sigma)=\dim(x)$, and thus the above construction produces a $C_n$-action on $\dim(x)$. 

The goal of this appendix is to prove the following claim:
\begin{prop}\label{prop:actontr}
    The equivalence $$\tr(x^{\otimes n};\sigma)\simeq \dim(x)$$ is $C_n$-equivariant, where the right hand side has the $C_n$-action described in \Cref{ex:actionontr}, and the left hand side has the $C_n$-action restricted from the usual $S^1$-action. 
\end{prop}
As usual, it suffices to prove this in the universal case, i.e. where $\C= \Cob_1$ and $x$ is the universal dualizable object. 

To prove this case, we observe the following:
\begin{obs}\label{obs:cobcospan}
    The symmetric monoidal functor $\Cob_1\to \cospan(\Ss^\fin)$ induces, on endomorphisms of the unit, a map which sends the circle (as a framed manifold) to the circle (as a homotopy type), and on automorphisms of those, it is given by  $(\id, 0): S^1\to S^1\times\Z/2$. 
\end{obs}
(this can be proved by explicitly describing the functor $\Cob_1\to \cospan(\Ss^\fin)$, using that $\Cob_1$ is essentially a category of cospans)

In particular, to prove that the $C_n$-actions agree, it suffices to do so after passing to $\cospan(\Ss^\fin)$, where we have more freedom. For example, we can try to identify the $C_n$-action \emph{on the trace of the universal $C_n$-action} (for some sense of universal). 

Consider the natural map $A\to \cospan(\Ss^\fin/A)$. By the construction above, it induces a natural map $LA\to (\Ss^\fin/A)^\simeq$, as the unit in $\cospan(\Ss^\fin/A)$ is $\empty$, and so its endomorphism space is precisely $(\Ss^\fin/A)^\simeq$. 

By naturality in $A$, the Yoneda lemma implies that this is completely determined  by the value at $S^1$ of the identity in $LS^1=\Map(S^1,S^1)$, which is, in turn, a certain map $S^1\to S^1$. 

Comparing $\Ss^\fin/S^1\simeq (\Ss^{B\Z})^{\fin}$,we find that this trace, in $\cospan((\Ss^{B\Z})^\fin)$ is the pushout of the following span:
\[\begin{tikzcd}
	\Z && \Z \\
	& \Z\coprod\Z
	\arrow["{(\id,\id)}", from=2-2, to=1-1]
	\arrow["{(\id,\sigma)}"', from=2-2, to=1-3]
\end{tikzcd}\] 
which one easily computes to be $\pt$ (here, $\sigma$ denotes the successor function). Now, the object $\pt$ corresponds to the identity map $S^1\to S^1$ under the equivalence $\Ss^{B\Z}\simeq \Ss_{/S^1}$, so this proves (by the Yoneda lemma) that the map $LA\to (\Ss^\fin/A)^\simeq$ we constructed with traces is simply the composite $$LA = (\Ss^\fin/A \times_{\Ss^\fin} \{S^1\})^\simeq \to (\Ss^\fin/A)^\simeq$$

\begin{proof}[Proof of \Cref{prop:actontr}]
By the discussion in \Cref{obs:cobcospan}, it suffices to prove this equivalence in the special case of $\C=\cospan(\Ss^\fin), x=\pt$. 

The composite functor $BC_n\to \cospan((\Ss^{BC_n})^\fin)\to \cospan(\Ss^\fin)$, where the second functor is the forgetful functor, classifies $S^1$ with the action induced from $C_n\subset S^1$. 

Thus it suffices to prove that the $BC_n$-action on $\tr(\pt;\sigma)\simeq S^1$ in endomorphisms of the unit in $\cospan((\Ss^{BC_n})^\fin)$ is the canonical one, or equivalently, the same statement for $\tr(S^1;\sigma)$ in $\cospan(\Ss^\fin/BC_n)$.

But we have identified the corresponding trace map $LBC_n\to (\Ss^\fin/BC_n)^\simeq$ with the ``canonical map'', and this canonical map is now a map of ordinary $2$-groupoids, namely: $$\map(B\Z,BC_n)\to (1\mathrm{Gpd}/BC_n)^\simeq$$ which one can simply fully identify by hand, and check the statement there. 
    \end{proof}

\bibliographystyle{alpha}
\bibliography{Biblio.bib}

\end{document}